\documentclass[11pt]{amsart}% Specifies the document style.
\usepackage{setspace}
\usepackage{amsmath}\allowdisplaybreaks
\usepackage{amscd,amssymb,stmaryrd,amstext,amsthm}
\usepackage{url}
\usepackage{subfigure}
\usepackage{enumitem}
\usepackage{graphicx}
\usepackage{color}
\usepackage{epstopdf}
\usepackage{caption}
\usepackage{multicol}
\usepackage{tikz}
\usepackage{pgfplots} 
\pgfplotsset{compat=newest}
\usetikzlibrary{decorations.pathreplacing,angles,quotes}
\usetikzlibrary{bending}
\usepackage{mathtools}
\usepackage[colorlinks]{hyperref}
\usepackage[all]{hypcap}
\usepackage[lite]{amsrefs}
\usepackage{amsfonts}
\usepackage[normalem]{ulem}

\newcounter{desccount}

\newcommand{\descref}[1]{\hyperref[#1]{#1}}

\numberwithin{equation}{section}

\newcommand{\sub}{\subseteq}

\newcommand{\Z}{\mathbb{Z}}
\newcommand{\R}{\mathbb{R}}
\newcommand{\C}{\mathbb{C}}

\newcommand{\Ga}{\alpha}

\newcommand{\Gd}{\delta}
\newcommand{\Ge}{\varepsilon}

\newcommand{\Gs}{\sigma}

\numberwithin{chap}{section}
\newtheorem{thm}{Theorem}
\numberwithin{thm}{section}
\newtheorem{conj}[thm]{Conjecture}
\newtheorem{prop}[thm]{Proposition}
\newtheorem{defn}[thm]{Definition}
\newtheorem{lem}[thm]{Lemma}

\newtheorem{cor}[thm]{Corollary}

\DeclarePairedDelimiter{\norm}{\lVert}{\rVert}
\DeclarePairedDelimiter{\abs}{\lvert}{\rvert}
\DeclarePairedDelimiter{\ip}{\langle}{\rangle}

\makeatletter
\let\oldabs\abs
\def\abs{\@ifstar{\oldabs}{\oldabs*}}
\let\oldnorm\norm
\def\norm{\@ifstar{\oldnorm}{\oldnorm*}}
\let\oldip\ip
\def\ip{\@ifstar{\oldip}{\oldip*}}
\makeatother

\begin{document}

\pagestyle{myheadings} \thispagestyle{empty} \markright{}
\title{Decoupling for smooth surfaces in $\mathbb R^3$}

\author{Jianhui Li and Tongou Yang}
\address[Jianhui Li]{Department of Mathematics, University of Wisconsin-Madison, Van Vleck Hall, 480 Lincoln Drive, Madison, WI 53706, United States}

\curraddr{Department of Mathematics, Northwestern University\\
Evanston, IL 60208, United States
}
\email{jianhui.li@northwestern.edu}

\address[Tongou Yang]{Department of Mathematics, The University of British Columbia, Vancouver, B.C. V6T 1Z2, Canada}
\curraddr{Department of Mathematics, The University of California, Los Angeles, CA 90095, United States
}
\email{tongouyang@math.ucla.edu}

\date{}
\maketitle

\begin{abstract}
     For each $d\geq 0$, we prove decoupling inequalities in $\mathbb R^3$ for the graphs of all bivariate polynomials of degree at most $d$ with bounded coefficients, with the decoupling constant uniform in the coefficients of those polynomials. As a consequence, we prove a decoupling inequality for (a compact piece of) every smooth surface in $\R^3$, which in particular solves a conjecture of Bourgain, Demeter and Kemp.
\end{abstract}

\section{Introduction}

Let $\phi:[-1,1]^2\to \R$ be a smooth function. For a tiny $\Gd>0$, consider the $\Gd$-neighbourhood $\mathcal N_\Gd$ of the graph of $\phi$ over $[-1,1]^2$ in $\R^3$. Our main goal in this paper is to prove the following statements. First, there is a suitable ``partition" of $\mathcal N_\Gd$ into almost rectangular boxes $\tau$ according to the curvature of the graph of $\phi$. Second, the following decoupling inequality holds: for every test function $f$ Fourier supported in $\mathcal N_\Gd$ and every $2\le p\le 4$,
$$
\norm{f}_{L^p(\R^3)}\lesssim_{\Ge} \Gd^{-\Ge} (\#\tau)^{\frac 1 2-\frac 1 p}\left(\sum_{\tau}\norm{\mathcal R_\tau f}^p_{L^p(\R^3)}\right)^{1/p},\quad \forall \Ge>0,
$$
where $\mathcal R_\tau f$ denotes the Fourier restriction of $f$ into the cap $\tau$. A more precise version is Theorem \ref{thm_conj} below.

Before we move on to the precise formulation of our decoupling theorem, we first give a brief literature review.

\subsection{Background}
In 2015, Bourgain and Demeter \cite{BD2015} proved a seminal $\ell^2$ decoupling inequality for every compact smooth hypersurface with positive Gaussian curvature. Later in \cite{BD2017} they proved a $\ell^p$ decoupling theorem for every compact smooth hypersurface with nonzero Gaussian curvature.

The general decoupling theory for smooth hypersurfaces of vanishing Gaussian curvature is still under-developed. Yet in $\R^2$, a smooth hypersurface reduces to a smooth curve, and its decoupling theory is relatively well-understood. As an initiation of this subject, a decoupling inequality for every analytic function defined on a compact interval is proved in \cite{BGLSX} . Later in Section 12.6 of \cite{Demeter2020}, Demeter proved a slightly more refined decoupling inequality for every analytic function defined on a compact interval, with the partition chosen to be adapted to the curvature of the function. Finally, in \cite{Yang2}, the second author proved a decoupling inequality for the family of all polynomials of degree at most $d$ and with coefficients bounded by $1$, with the decoupling constant depending only on $d$ but not the individual polynomial. This, together with a brute-force Taylor approximation (which we will see in Section \ref{sec_conj} below), implies a decoupling inequality for every single smooth function defined on a compact interval, further generalising the result in \cite{Demeter2020}.

In higher dimensions, even $\R^3$, the situation becomes much more subtle. In \cite{BDK2019}, Bourgain, Demeter and Kemp proposed a conjecture. To formulate the conjecture, we first have the following definition.

\begin{defn}[$\Gd$-flatness]\label{defn_flatness}
Let $n\ge 1$ and $\phi:\R^n\to \R$ be smooth. We say $\phi$ is $\Gd$-flat over a set $S\sub \R^n$, or alternatively $S$ is $(\phi,\Gd)$-flat, if 
\begin{equation*}
    \sup_{u,v\in S} |\phi(v)-\phi(u)-\nabla \phi(u)\cdot(v-u)|\leq \Gd.
\end{equation*}
\end{defn}
For some elementary properties of $\Gd$-flatness, the reader may refer to the appendix of \cite{LiYang}. It is worth noting that $\Gd$-flat sets are also considered in \cites{BNW,NSW1993}.

\begin{conj}[Bourgain-Demeter-Kemp \cite{BDK2019}, Conjecture 7.4]\label{conj_BDK}
Given a real-analytic function $\phi:[-1,1]^2\to \R$ and $\Gd>0$, we can find a partition $\mathcal P_\Gd=\mathcal P_\Gd(\phi)$ of $[-1,1]^2$ into subsets $S$ on which $\phi$ is $\Gd$-flat. Moreover, for every $2\le p\le 4$, we have the following $\ell^p(L^p)$ decoupling inequality: for every $f:\R^3\to \C$ Fourier supported in the $\Gd$-neighbourhood of the graph of $\phi$, we have
$$
\norm {f}_{L^p(\R^3)}\leq C_{\phi,\Ge}\Gd^{-\Ge} \# \mathcal P_\Gd ^{\frac 1 2-\frac 1 p}\left(\sum_{S\in \mathcal P_\Gd}\norm{f_S}_{L^p(\R^3)}^p\right)^{\frac 1 p},\,\, \text{for every $\Ge>0$}.
$$
Here and throughout this paper, for $f:\R^3\to \C$ and $S\sub \R^2$, we denote by $f_S$ the Fourier restriction of $f$ to the strip $S\times \R$, namely, $f_S$ is defined by the relation
\begin{equation}\label{eqn_Fourier_restriction_strip}
    \widehat {f_S}(x,y,z)=\hat f(x,y,z)1_S(x,y).
\end{equation}
\end{conj}

Only partial results have been established before this paper. In \cite{BDK2019}, they proved decoupling inequalities for all real-analytic surfaces of revolution in $\R^3$. Later Kemp \cite{Kemp} proved decoupling inequalities for surfaces with constantly zero Gaussian curvature but without umbilical points, and in a recent work \cite{KempJFA} he is also able to prove decoupling inequalities for a broad class of $C^5$ surfaces in $\R^3$ lacking planar points. Also recently, the authors proved in \cite{LiYang} a decoupling inequality for every mixed-homogeneous polynomial in $\R^3$. Note that none of the previous partial results implies one another.

These hypersurfaces of vanishing Gaussian curvature are also of interest in other related problems in harmonic analysis, including the restriction problem and the $L^p$ improving estimates for averages. For example, the reader may refer to \cites{CKZ2013,GMW2024,IM2016,M2014,NSW1993,Oberlin2012,PS1997,SS2021}.

\subsection{Main results}
In this article, we propose a solution to Conjecture \ref{conj_BDK}. We first state the main result.

\begin{thm}[Solution to Conjecture \ref{conj_BDK}]\label{thm_conj}
Let $\phi:[-1,1]^2\to \R$ be a smooth function. Then for every $0<\Gd<1$, there is a family $\mathcal P_\Gd=\mathcal P_\Gd(\phi)$ of rectangles $T$ covering $[-1,1]^2$, such that the following statements hold:
\begin{enumerate}
    \item $\mathcal P_\Gd$ has $\Ge$-bounded overlap in the sense that  $\sum_{T\in \mathcal P_\Gd}1_{T} \lesssim_ {\phi,\Ge}\Gd^{-\Ge}$ for every $\Ge>0$.
    \item $\phi$ is $\Gd$-flat over each $T\in \mathcal P_\Gd$.
    \item For every $2\le p\le 4$, we have the following $\ell^p(L^p)$ decoupling inequality: for any function $f:\R^3\to \C$ with Fourier support in the $\Gd$-vertical neighbourhood of the graph of $\phi$ above $[-1,1]^2$, we have
\begin{equation}\label{eqn_decoupling_l4_conj}
    \norm {f}_{L^p(\R^3)}\lesssim_{\phi,\Ge}\Gd^{-\Ge}\# \mathcal P_\Gd^{\frac 1 2-\frac 1 p}\left(\sum_{T\in \mathcal P_\Gd}\norm{f_T}_{L^p(\R^3)}^p\right)^{\frac 1 p},\,\, \text{for every $\Ge>0$}.
\end{equation}
    If, in addition, $\phi$ is convex (or concave) on $[-1,1]^2$ , then \eqref{eqn_decoupling_l4_conj} can be strengthened to the $\ell^2(L^p)$ decoupling inequality: for every $2\le p\le 4$,
\begin{equation}\label{eqn_decoupling_l2_conj}
   \norm {f}_{L^p(\R^3)}\lesssim_{\phi,\Ge}\Gd^{-\Ge}\left(\sum_{T\in \mathcal P_\Gd}\norm{f_T}_{L^p(\R^3)}^2\right)^{\frac 1 2},\,\, \text{for every $\Ge>0$}.
\end{equation}  
\end{enumerate}
\end{thm}
Note that this theorem is slightly different from Conjecture \ref{conj_BDK}. First, Theorem \ref{thm_conj} works for a smooth function which may not necessarily be analytic, so our result is stronger in this sense. Also, we have \eqref{eqn_decoupling_l2_conj} in the convex case, which was not explicitly mentioned in \cite{BDK2019}.

Second, instead of constructing a strict partition of $[-1,1]^2$ into suitable subsets, we construct a $\Ge$-boundedly overlapping family of rectangles covering $[-1,1]^2$. We choose to use rectangles mainly because they are easier to deal with, and more importantly, rectangles are the main objects of concern in restriction and Kakeya type problems. The $\Ge$-bounded overlap would not do any harm since in \eqref{eqn_decoupling_l4_conj} we already lose a factor of the form $C_{\phi,\Ge}\Gd^{-\Ge}$, and this is a common feature in decoupling inequalities (in some very special case the $\Ge$-loss can be improved; see \cite{GMW2024}). Nevertheless, a slight modification of the proof of Theorem \ref{thm_conj} also gives a strict partition of $[-1,1]^2$ into connected subsets $S$. As can be seen in the proof, the only source of overlapping is the enlargement of certain sets to rectangles. These sets are the intersections of sublevel sets of some polynomials and parallelograms. We can instead keep these sets, which form a disjoint partition of the set we want to decouple. By doing so in every iteration, we get disjoint connected subsets $S$ whose boundaries are algebraic curves. In particular, Conjecture \ref{conj_BDK} holds.

Heuristically, the main geometric reason for Theorem \ref{thm_conj} to hold is that $\mathcal P_\Gd$ is a collection of ``maximal" $(\phi,\Gd)$-flat rectangles $T$, meaning that $MT$ is not $(\phi,\Gd)$-flat if $M$ is significantly larger than an absolute constant. If $\phi$ is convex, then $\mathcal P_\Gd$ is essentially unique, and thus a ``maximum" collection; see \cites{BNW,NSW1993}. If $\phi$ is non-convex, then there are significantly different ``maximal" collections $\mathcal P_\Gd$ but not all of them give rise to a $\ell^2(L^p)$ decoupling inequality as in Theorem \ref{thm_conj}. A typical example is when $\phi(x,y)=xy$ and $\mathcal P_\Gd$ is a tiling of $[-1,1]^2$ by axis-parallel rectangles of dimensions $\Gd\times 2$, which are ``maximal" but for which \eqref{eqn_decoupling_l4_conj} fails to hold except for the trivial case $p=2$, as can be seen by taking the test function $f$ to be the characteristic function of the $\delta$-neighborhood of the $x$-axis.

A close cousin of Theorem \ref{thm_conj} is the following uniform decoupling inequality for all bivariate polynomials of degree at most $d$ and with coefficients bounded by $1$.
\begin{thm}[Uniform decoupling theorem]\label{thm_main_uniform_decoupling}
Let $d\geq 0$. Then for every polynomial $\phi:\R^2\to \R$ of degree at most $d$ and with coefficients bounded by $1$, for every $0<\Gd<1$, there is a family $\mathcal P_\Gd=\mathcal P_\Gd(\phi,d)$ of rectangles $T$ covering $[-1,1]^2$ such that the following statements hold:
\begin{enumerate}
    \item $\mathcal P_\Gd$ has $\Ge$-bounded overlap in the sense that $\sum_{T\in \mathcal P_\Gd}1_{T} \lesssim_{d,\Ge}\Gd^{-\Ge}$ for every $\Ge>0$.
    \item $\phi$ is $\Gd$-flat over each $T\in \mathcal P_\Gd$.
    \item For every $2\le p\le 4$, we have the following $\ell^p(L^p)$ decoupling inequality: for any function $f:\R^3\to \C$ with Fourier support in the $\Gd$-vertical neighbourhood of the graph of $\phi$ above $[-1,1]^2$, 
\begin{equation}\label{eqn_decoupling_l4}
    \norm {f}_{L^p(\R^3)}\lesssim_{d,\Ge}\Gd^{-\Ge}\# \mathcal P_\Gd^{\frac 1 2-\frac 1 p}\left(\sum_{T\in \mathcal P_\Gd}\norm{f_T}_{L^p(\R^3)}^p\right)^{\frac 1 p},\,\, \text{for every $\Ge>0$}.
\end{equation}
If, in addition,  $\phi$ is convex (or concave) on $[-1,1]^2$, then \eqref{eqn_decoupling_l4} can be strengthened to the $\ell^2(L^p)$ decoupling inequality:
\begin{equation}\label{eqn_decoupling_l2}
   \norm {f}_{L^p(\R^3)}\lesssim_{d,\Ge}\Gd^{-\Ge}\left(\sum_{T\in \mathcal P_\Gd}\norm{f_T}_{L^p(\R^3)}^2\right)^{\frac 1 2},\,\, \text{for every $\Ge>0$}.
\end{equation} 
\end{enumerate}
The implicit constants depend on $d,\Ge$ but not $\phi,\Gd,f$.
\end{thm}

In Theorem 2.1 of \cite{KempJFA}, Kemp formulated a mixed $\ell^2$ and $\ell^p$ decoupling inequality adapted to positively curved and negatively curved parts respectively. Indeed, our method also naturally gives rise to such kind of inequality. In particular, we will show the $\ell^2(L^p)$ decoupling inequality with slightly relaxed condition $
\det D^2 \phi > -\delta^2$ on $[-1,1]^2$. This is crucial to the proof of the $\ell^2(L^p)$ decoupling for convex smooth functions as it allows certain errors in the Hessian determinant after the Taylor approximation. See Sections \ref{sub_Taylor} and \ref{sub_convex}.

The vast majority of this article is devoted to the proof of Theorem \ref{thm_main_uniform_decoupling}.

\subsection{Main idea of proof}
In what follows, we outline the proof of the case $p=4$, the other cases $2\leq p < 4$ follows similarly.

The following results of Bourgain and Demeter will serve as the cornerstone of our argument.
\begin{thm}[Bourgain-Demeter, \cites{BD2015,BD2017}\label{thm_Bourgain_Demeter}]
Let $\phi:[-1,1]^2\to \R$ be a $C^3$-function with nonsingular Hessian, and suppose $\norm{\phi}_{C^3}\leq C_0$. For every $0<\Gd<1$, denote by $\mathcal P_\Gd$ a tiling of $[-1,1]^2$ by squares $T$ of side length $\Gd^{1/2}$. Then for any $2\leq p \leq 4$, any function $f:\R^3\to \C$ with Fourier support in the $\Gd$-vertical neighbourhood of the graph of $\phi$ above $[-1,1]^2$, we have
\begin{equation}\label{eqn_BD_decoupling_l4}
    \norm {f}_{L^p(\R^3)}\lesssim_{C_0,\Ge}\Gd^{-\Ge}\# \mathcal P_\Gd^{\frac{1}{2} - \frac{1}{p}}\left(\sum_{T\in \mathcal P_\Gd}\norm{f_T}_{L^p(\R^3)}^p\right)^{\frac 1 p},\,\, \text{for every $\Ge>0$}.
\end{equation}
If, in addition,  $\phi$ is convex (or concave) on $[-1,1]^2$, then \eqref{eqn_BD_decoupling_l4} can be strengthened to
\begin{equation}\label{eqn_BD_decoupling_l2}
    \norm {f}_{L^p(\R^3)}\lesssim_{C_0,\Ge}\Gd^{-\Ge}\left(\sum_{T\in \mathcal P_\Gd}\norm{f_T}_{L^p(\R^3)}^2\right)^{\frac 1 2},\,\, \text{for every $\Ge>0$}.
\end{equation}
\end{thm}
The $\ell^2(L^p)$-inequality \eqref{eqn_BD_decoupling_l2} follows from Section 7 of \cite{BD2015} when $n=3$. The $\ell^p(L^p)$-inequality \eqref{eqn_BD_decoupling_l4} follows from Theorem 1.1 (see also Remark 2.3)  of \cite{BD2017} when $n=3$, combined with a standard Pramanik-Seeger type iteration (see \cite{PS2007}). The fact that the implicit constants only depend on $C_{0},\Ge$ but not individual choice of $\phi$ is not explicitly stated in \cites{BD2015,BD2017} but can be deduced from the proofs. (See Section 5 of \cite{Yang2} for a rigorous proof in the 2D case; higher dimensional cases are the same.)

\subsubsection{Passing to polynomials and projecting to $\R^2$.}
The reason we prove Theorem \ref{thm_main_uniform_decoupling} first is that polynomials have nice algebraic properties that can be utilised.  As described in Section \ref{sec_conj}, we use a trivial decoupling and Taylor approximation of degree $1/\varepsilon$ to reduce the problem to polynomial cases. Given a bivariate polynomial $\phi$, in view of Theorem \ref{thm_Bourgain_Demeter}, the major difficulty of Conjecture \ref{conj_BDK} is the presence of the vanishing set of $\det D^2 \phi$. In the cases considered in \cite{KempJFA} and \cite{LiYang}, $\{|\det D^2 \phi|<\delta\}$ is a neighborhood of the graph of some polynomial. Both papers used projections to decouple such sets in the setting of $\R^2$. Notably, these kinds of projection arguments are used repeatedly in similar decoupling problems, e.g. \cites{BD2015,BDG2016,GOZZ,GZ2020}. Although a projection forgets certain geometry of the surface, the decoupling of the projected sets, as a subset of $\R^2$, is easier to study. See \cites{BGLSX,Demeter2020,Yang2}.

\subsubsection{Generalised 2D uniform decoupling.}
For now, we assume $\phi$ to be a polynomial of degree at most $d$ and coefficients bounded by $1$, and let $P = \det D^2 \phi$. As mentioned above, we first decouple the projected sets $\{|P| <\delta\} \subset \R^2$ into rectangles. We call this generalised 2D uniform decoupling. The special case where $|P_y| \gtrsim 1$ is more manageable because it can be approximated by polynomial-like functions locally. In fact, by Pramanik-Seeger type iteration in \cite{PS2007} and uniform decoupling on polynomial-like functions in \cite{Yang2}, we first prove this special case. To prove the general case, we need to dyadically decompose $\nabla P$, which is one degree lower than $P$. One may assume by induction on the degree of the polynomials that the dyadic set $|P_x|\sim \sigma $ can be decoupled into rectangles, similar for $P_y$. Then, a certain invertible linear map sending these rectangles to the unit square allows one to reduce the problem to the special case $|P_y|\gtrsim 1$. We then conclude by using the decoupling result for the special case.

\subsubsection{Small Hessian determinant and cylindrical decoupling.}
By generalised 2D uniform decoupling, we are able to localize $\phi$ to rectangles with small Hessian determinant. Rescaling these rectangles to the unit square and we are left with decoupling a rescaled function which we denote by $\Tilde{\phi}$. In the case of 2D decoupling for polynomials as in \cite{Yang2}, this would be already enough to conclude an induction on scale argument because then the rescaled function automatically has small coefficients. However, in the case of bivariate polynomial, it may happen that the Hessian determinant of $\tilde{\phi}$ is $\lesssim M^{-1} \ll 1$ over the unit square (for some large $M$ to be determined), but some coefficient of $\tilde \phi$ has magnitude $\sim 1$. To resolve this, we prove a proposition in Section \ref{sec_small_hessian} that says such polynomials admit the form $$\tilde{\phi} \circ \rho = A(x) + M^{-\alpha}B(x,y)y$$
for some rotation $\rho$, $\alpha=\alpha(d) \in (0,1)$ and polynomials $A$, $B$ with $O_d(1)$ coefficients.

Now, applying a cylindrical decoupling to $\Tilde{\phi} \circ \rho$ at the scale $M^{-\alpha}$, we get $(\Tilde{\phi},M^{-\alpha})$-flat rectangles. Thus, by rescaling these rectangles to the unit square, we get polynomials (denoted by $\bar{\phi}$) with $O_d(M^{-\alpha})$ coefficients. 

\subsubsection{Induction on scales} We then proceed by induction on the scale $\delta$ and decouple $M^{\alpha} \bar{\phi}$ into $(M^{\alpha} \bar{\phi},M^\alpha \delta)$-flat rectangles. Equivalently, these rectangles are $(\bar{\phi},\delta)$-flat. Rescaling back, we get a decoupling of the region with small Hessian determinant into $(\phi,
\delta)$-flat rectangles. The part with large Hessian determinant can be decoupled using Theorem \ref{thm_Bourgain_Demeter}, with decoupling constant $C_{\varepsilon,M}M^{\varepsilon}$. By choosing $M=M(\varepsilon)$ large enough, we conclude Theorem \ref{thm_main_uniform_decoupling}.

\subsection{Open problems}
With Theorems \ref{thm_conj} and \ref{thm_main_uniform_decoupling}, we would also like to ask the following questions.
\begin{enumerate}
    \item Can we improve the $\Ge$-bounded overlap to $C_\phi$ (or $C_d$)-bounded overlap? We believe our method actually gives this conclusion, but it would require much more detailed analysis.
   
    \item Higher dimensional generalisations. Can we formulate and prove versions of Theorems \ref{thm_conj} and \ref{thm_main_uniform_decoupling} in $\R^n$ for $n\ge 4$? For the previous open problem we believe there are just some technical difficulties, but for this open problem our current decoupling analysis fails substantially. There are many difficulties we have to face, but the major problem lies in an upgrade of Section \ref{sub_combining_decoupling}, where we applied H\"older's inequality to combine a 2D lower dimensional $\ell^2(L^p)$ decoupling at an intermediate scale and a $\ell^4(L^p)$ decoupling at a finer scale in 3D to obtain a final $\ell^4(L^p)$ decoupling inequality. The fact that there is no cardinality term $\#\mathcal P_\Gd$ in $\ell^2(L^p)$ decoupling inequalities is the key savior of our argument. In the passage from 3D to 4D, this in general will require us to combine a 3D lower dimensional $\ell^4(L^{4})$ (or perhaps $\ell^4(L^{10/3})$) decoupling at an intermediate scale and a $\ell^{10/3}(L^{10/3})$ decoupling at a finer scale in 4D; however, since in $\ell^4(L^{4})$ decoupling inequalities in 3D we inevitably have a cardinality term $\#\mathcal P_\Gd$ (see \eqref{eqn_decoupling_l4}), this will make the combination much more subtle. 
    
    Another difficulty will be the higher dimensional analogues of Theorem \ref{thm_small_hessian} and Proposition \ref{prop_delta_nbhd}, as singularities of polynomials in higher dimensions can be so bad that much more refined algebraic geometry may be needed.
\end{enumerate}

\subsection{Notation.}
\begin{enumerate}

\item \label{notation_vertical_nbhd} For any set $A\sub \R^2$ (or $A\sub \R$), any continuous function $\phi:A\to \R$ and any $\Gd>0$, denote the (vertical) $\Gd$-neighbourhood of the graph of $\phi$ above $A$ by
$$
\mathcal N^\phi_{\Gd}(A)=\{(x,y,z):(x,y)\in A, |z-\phi(x,y)|<\Gd\},
$$
(with the obvious modification for $A\sub \R$). 

\item In addition to the notation $f_S$ introduced in \eqref{eqn_Fourier_restriction_strip}, if $S$ lives in the same Euclidean space as the domain of $f$ does, we write $\mathcal R_S f$ to be the Fourier restriction of $f$ to $S$, namely, $\widehat {\mathcal R_S f}(x)=\hat f (x)1_S(x)$.

\item \label{notation_equivalent} Let $A\sub \R^2$ be a parallelogram and $B\sub \R^2$. $B$ is said to be {\it equivalent} to $A$ if for some absolute constant $C\ge 1$ we have $C^{-1}A\sub B\sub CA$, where $CA$ means the dilation of $A$ by a factor of $C$ with respect to the centre of $A$.

\item We use the standard notation $a=O_C(b)$, or $|a|\lesssim_C b$ to mean there is a constant $C$ such that $|a|\leq Cb$. In many cases, such as results related to polynomials, the constant $C$ will be taken to be absolute or depend on the polynomial degree only, and so we may simply write $a=O(b)$ or $|a|\lesssim b$.
%\item The notion of $\Ge$-overlap has been defined in the statements of Theorems \ref{thm_main_uniform_decoupling} and \ref{thm_conj}. We used the terminology ``$\Ge$-overlap'' since the overlapping bound depends on an arbitrarily small $\Ge$. In some cases related to polynomials, such as Theorem \ref{thm_2D_general_uniform}, the rectangles can be chosen so that the overlapping bound does not depend on $\Ge$ but only the degree $d$ of the polynomial. In this case we will simply say that they are {\it boundedly overlapping} or they have {\it bounded overlap}.

\item For a polynomial, we use the term ``linear terms" to denote the sum of all its terms of degree at most $1$.

\end{enumerate}

{\it Outline of the paper.} In Section \ref{sec_conj} we use an elementary argument and a simple brute-force Taylor approximation to reduce Theorem \ref{thm_conj} to a superficially weaker version of Theorem \ref{thm_main_uniform_decoupling}, Theorem \ref{thm_main_uniform_decoupling_eps}. In Section \ref{sec_ingredients}, we state the main ingredients we will need to prove Theorem \ref{thm_main_uniform_decoupling_eps}. In Section \ref{sec_2D_uniform}, we prove the first and the most important ingredient, the generalised 2D uniform decoupling for polynomials. We then proceed to prove Theorem \ref{thm_main_uniform_decoupling_eps} in Section \ref{sec_proof_main} and postpone the calculations of some elementary facts about small Hessian determinants to Section \ref{sec_small_hessian}. Lastly, Section \ref{sec_appendix} is an appendix, where we state a simple property of polynomials we use extensively in this article.

{\it Acknowledgement.} The first author would like to thank his advisor Betsy Stovall for her advice and help throughout the project. He would also like to thank Po Lam Yung for his invaluable guidance.

The second author would like to thank Malabika Pramanik, Joshua Zahl and Po Lam Yung for their guidance and kindness throughout the past years. 

Both authors would like to thank Jaume de Dios, Shaoming Guo, D\'ominique Kemp, Terence Tao and Hong Wang for useful discussions related to this article. We also especially thank Betsy Stovall and Joshua Zahl for suggesting the Taylor approximation and the removal of $\Ge$-dependence of partitions in Section \ref{sec_conj}, respectively. We would like to thank the anonymous referee for their careful reading of our manuscript and their insightful suggestions. The first author was partially supported by NSF DMS-1653264.

\section{Proof of Conjecture}\label{sec_conj}

The ultimate goal of this section is the proof of Theorem \ref{thm_conj}, and thus Conjecture \ref{conj_BDK}.

\subsection{$\Ge$-dependent partitions}
For technical reasons, we will first consider the following superficially weaker versions of Theorems \ref{thm_conj} and \ref{thm_main_uniform_decoupling} when the partitions depend on $\Ge$ as well.

\begin{thm}[Theorem \ref{thm_conj} with $\Ge$-dependent partitions]\label{thm_conj_eps}
Let $\phi:[-1,1]^2\to \R$ be a smooth function. Then for every $\Ge>0$ and every $0<\Gd<1$, there is a family $\mathcal P_\Gd=\mathcal P_\Gd(\phi,\Ge)$ of rectangles $T$ covering $[-1,1]^2$, such that the following statements hold:
\begin{enumerate}
    \item $\mathcal P_\Gd$ has $\Ge$-bounded overlap in the sense that $\sum_{T\in \mathcal P_\Gd}1_{T} \lesssim_{\phi,\Ge}\Gd^{-\Ge}$.
    \item $\phi$ is $\Gd$-flat over each $T\in \mathcal P_\Gd$.
    \item We have the following $\ell^p(L^p)$ decoupling inequality: for any function $f:\R^3\to \C$ with Fourier support in the $\Gd$-vertical neighbourhood of the graph of $\phi$ above $[-1,1]^2$, we have
\begin{equation}\label{eqn_decoupling_l4_conj_eps}
    \norm {f}_{L^p(\R^3)}\lesssim_{\phi,\Ge}\Gd^{-\Ge}\# \mathcal P_\Gd^{\frac 1 2- \frac 1 p}\left(\sum_{T\in \mathcal P_\Gd}\norm{f_T}_{L^p(\R^3)}^p\right)^{\frac 1 p}.
\end{equation}
If, in addition,  $\phi$ is convex (or concave) on $[-1,1]^2$, then \eqref{eqn_decoupling_l4} can be strengthened to the $\ell^2(L^p)$ decoupling inequality:
\begin{equation}\label{eqn_decoupling_l2_conj_eps}
   \norm {f}_{L^p(\R^3)}\lesssim_{\phi,\Ge}\Gd^{-\Ge}\left(\sum_{T\in \mathcal P_\Gd}\norm{f_T}_{L^p(\R^3)}^2\right)^{\frac 1 2}.
\end{equation} 
\end{enumerate}
\end{thm}

\begin{thm}[Theorem \ref{thm_main_uniform_decoupling} with $\Ge$-dependent partitions]\label{thm_main_uniform_decoupling_eps}
Let $d\geq 0$. Then for every $\Ge>0$, every polynomial $\phi:\R^2\to \R$ of degree at most $d$ and with coefficients bounded by $1$, and every $0<\Gd<1$, there is a family $\mathcal P_\Gd=\mathcal P_\Gd(\phi,d,\Ge)$ of rectangles $T$ covering $[-1,1]^2$ such that the following statements hold:
\begin{enumerate}
    \item $\mathcal P_\Gd$ has $\Ge$-bounded overlap in the sense that $\sum_{T\in \mathcal P_\Gd}1_{T} \lesssim_{d,\Ge}\Gd^{-\Ge}$.
    \item $\phi$ is $\Gd$-flat over each $T\in \mathcal P_\Gd$.
    \item We have the following $\ell^p(L^p)$ decoupling inequality: for any function $f:\R^3\to \C$ with Fourier support in the $\Gd$-vertical neighbourhood of the graph of $\phi$ above $[-1,1]^2$, 
\begin{equation}\label{eqn_decoupling_l4_eps}
    \norm {f}_{L^p(\R^3)}\lesssim_{d,\Ge}\Gd^{-\Ge}\# \mathcal P_\Gd^{\frac 1 p}\left(\sum_{T\in \mathcal P_\Gd}\norm{f_T}_{L^p(\R^3)}^p\right)^{\frac 1 p}.
\end{equation}
If, in addition, we have $\det D^2 \phi> -\delta^{2}$  on $[-1,1]^2$, then \eqref{eqn_decoupling_l4} can be strengthened to the $\ell^2(L^p)$ decoupling inequality:
\begin{equation}\label{eqn_decoupling_l2_eps}
   \norm {f}_{L^p(\R^3)}\lesssim_{d,\Ge}\Gd^{-\Ge}\left(\sum_{T\in \mathcal P_\Gd}\norm{f_T}_{L^p(\R^3)}^2\right)^{\frac 1 2}.
\end{equation} 
\end{enumerate}
The implicit constants depend on $d,\Ge$ but not $\phi,\Gd,f$.
\end{thm}
Notice that to prove \eqref{eqn_decoupling_l2_conj_eps} it suffices to assume a slightly weaker assumption $\det D^2\phi>-\Gd^2$ than convexity; thus the result is slightly stronger. This is used in the proof of Theorem \ref{thm_main_uniform_decoupling_eps} and in turn Theorem \ref{thm_main_uniform_decoupling}; see Section \ref{sub_Taylor} below.

\subsection{Removal of $\Ge$-dependence of partitions}
Now using an elementary argument, we will show that Theorem \ref{thm_conj_eps} implies Theorem \ref{thm_conj} and that Theorem \ref{thm_main_uniform_decoupling_eps} implies Theorem \ref{thm_main_uniform_decoupling}. 

\begin{proof}
Assume we have Theorem \ref{thm_conj_eps}. Let $\Ge=\Ge(\Gd)$ be a function such that $\Ge\searrow 0$ as $\Gd\searrow 0$, with the rate of decay to be determined.

Now with each $\Gd>0$, by Theorem \ref{thm_conj_eps}, we may find a family $\mathcal P_\Gd=\mathcal P_\Gd(\phi,\Ge(\Gd))$ of rectangles $T$ satisfying the prescribed properties. Note that by definition, $\mathcal P_\Gd$ depends on $\Gd$ only.

We then show that \eqref{eqn_decoupling_l4_conj_eps} implies \eqref{eqn_decoupling_l4_conj}. Let $\Ge>0$ be arbitrary. 

If $\Ge<2\Ge(\Gd)$, then using the monotonicity of the fixed function $\Ge(\Gd)$, $\Gd$ is bounded below by a positive number depending on $\Ge$ only. In this case, we have a trivial decoupling inequality, with the constant depending on $\phi,\Ge$.

If $\Ge\ge 2\Ge(\Gd)$, then it suffices to find a suitable $C'_{\phi,\Ge}$ such that
$$
C_{\phi,\Ge(\Gd)}\Gd^{-\Ge(\Gd)}\le C'_{\phi,\Ge}\Gd^{-\Ge},\,\,\text{for all }\Ge>0,\,\Gd>0,
$$
where $C_{\phi,\Ge}$ is the implicit constant in \eqref{eqn_decoupling_l4_conj_eps}. Since $\Ge\ge 2\Ge(\Gd)$, we have $\Gd^{-\Ge(\Gd)}\le \Gd^{-\Ge/2}$, and thus it suffices to show that
\begin{equation}\label{eqn_eps_delta}
    C_{\phi,\Ge(\Gd)}\Gd^{\Ge/2}\leq C'_{\phi,\Ge},\,\,\text{for all }\Ge>0,\,\Gd>0.
\end{equation}
We may assume $C_{\phi,\Ge}\nearrow \infty$ as $\Ge\searrow 0$. If we choose $\Ge(\Gd)$ to decrease slowly enough as $\Gd\searrow 0$, then we may have $C_{\phi,\Ge(\Gd)}\leq \log \Gd^{-1}$ for all $\Gd$ small enough. Then for all $\Ge>0$, we have 
$$
\lim_{\Gd\to 0^+}C_{\phi,\Ge(\Gd)}\Gd^{\Ge/2}\leq \lim_{\Gd\to 0^+}\Gd^{\Ge/2}\log \Gd^{-1}=0.
$$
Thus, by choosing a suitable constant $C'_{\phi,\Ge}$, we have \eqref{eqn_eps_delta}, and thus \eqref{eqn_decoupling_l4_conj}. The $\Ge$-bounded overlap follows from the same proof. In the convex case, the implication from \eqref{eqn_decoupling_l2_conj_eps} to \eqref{eqn_decoupling_l2_conj} also follows from the same proof. Thus we have Theorem \ref{thm_conj}.

The implication from Theorem \ref{thm_main_uniform_decoupling_eps} to Theorem \ref{thm_main_uniform_decoupling} follows from the same proof as above.
\end{proof}

\subsection{Theorem \ref{thm_main_uniform_decoupling_eps} implies Theorem \ref{thm_conj_eps}}\label{sub_Taylor}
Assuming Theorem \ref{thm_main_uniform_decoupling_eps}, we now prove Theorem \ref{thm_conj_eps}, which in turn implies Theorem \ref{thm_conj} by the subsection right above. The idea is a standard Taylor polynomial approximation.

\begin{proof}
Given $\phi\in C^\infty([-1,1]^2)$. We will first prove the general case without assuming $\phi$ is convex.

Given $\Ge>0$ and $\Gd>0$, we first partition $[-1,1]^2$ into squares $Q$ of side length $\Gd^{\Ge}$. Since we allow a loss of $\Gd^{-\Ge}$ in the decoupling inequality, by triangle and H\"older's inequalities it suffices to decouple each $Q$ into $(\phi,\Gd)$-flat rectangles. 

Let $d$ be the smallest integer greater than or equal to $\Ge^{-1}$. For each $Q$, we may approximate $\phi$ by its $d$-th degree Taylor polynomial $P_Q$, which depends on $\Ge,\Gd$. The error is at most 
$$
C_d\sup_{[-1,1]^2}|D^{d+1}\phi(x,y)|(\Gd^{\Ge})^{d+1}\leq C_d\sup_{[-1,1]^2}|D^{d+1}\phi(x,y)|\Gd^{1+\Ge}.
$$
For $\Gd$ small enough (depending on $\Ge$), the above error will be less than $\Gd/2$. Hence, to find a cover of $Q$ by $(\phi,\Gd)$-flat rectangles, it suffices to find a cover of $Q$ by $(P_Q,\Gd/2)$-flat rectangles.

Since all coefficients of $P_Q$ are bounded by a large constant depending on $\phi,\Ge$ only, we may normalise $P_Q$ to $\tilde P_Q$ with maximum coefficient having magnitude $1$, and apply Theorem \ref{thm_main_uniform_decoupling_eps} to $\tilde P_Q$ find a cover $\mathcal P_{\Gd,Q}$ of $Q$ that satisfies the prescribed properties. The final collection $\mathcal P_\Gd$ is simply defined to be the union of the collections $\cup_Q \mathcal P_{\Gd,Q}$. The overlap between rectangles covering different $Q$'s is trivially bounded, since each covering rectangle of $Q$ is a subset of $2Q$, and $2Q$'s have bounded overlap.

Although $P_Q$ depends on $\Gd$, our main uniform decoupling Theorem \ref{thm_main_uniform_decoupling_eps} ensures that there is a uniform bound of all decoupling constants as $Q$ varies, and this uniform bound is independent of $\Gd$. Thus, the final decoupling inequality \eqref{eqn_decoupling_l2_conj_eps} follows immediately by the triangle and H\"older's inequalities.

Lastly, in the special case when $\phi$ is convex (the concave case is symmetric), by choosing the degree of $P_Q$ to be greater than $\Ge^{-3}$ if necessary, we may approximate $\phi$ by $P_Q$ such that $\det D^2 P_Q\geq -\Gd^3$. By normalising $P_Q$ to $\tilde P_Q$ with maximum coefficient having magnitude $1$ and choosing $\Gd$ small enough, we also have $1\gtrsim_{\Ge}\det D^2 \tilde P_Q\geq -\Gd^2$. In this case, we get uniform $\ell^2$ decoupling inequalities from Theorem \ref{thm_main_uniform_decoupling_eps}. This finishes the proof.
\end{proof}

Therefore, all that is left is the proof of Theorem \ref{thm_main_uniform_decoupling_eps}.

\section{Ingredients of the proof of Theorem \ref{thm_main_uniform_decoupling_eps}}\label{sec_ingredients}
All implicit constants in this section are assumed to depend on the degree $d$.

The proof of the main theorem has the following main ingredients: induction on scales, a 2D uniform decoupling theorem for polynomials, a ``Small Hessian Theorem", and a simple projection lemma.

\subsection{Generalised 2D uniform decoupling for polynomials}
We need the following uniform decoupling theorem for bivariate polynomials.
\begin{thm}\label{thm_2D_general_uniform}
For each $d\geq 0$ and $\varepsilon>0$, there is a constant $C_{d,\varepsilon}$ such that the following holds. For any polynomial $P: \R^2 \to \R$ of degree at most $d$ and with coefficients bounded by $1$, any $0<\delta<1$, there exists a cover $\mathcal T_\Gd$ of the set
$$
\{(x,y)\in [-1,1]^2:|P(x,y)|<\delta\}
$$
by rectangles $T$ such that the following holds:
\begin{enumerate}
    \item $|P|\lesssim \Gd$ in each $T$.
    \item $\mathcal T_\Gd$ has bounded overlap in the sense that $\sum_{T\in \mathcal T_\Gd}1_{100T}\lesssim 1$.
    \item The following $\ell^2(L^p)$ decoupling inequality holds: for any $2\le p\le 6$ and any function $f:\R^2\to \C$ Fourier supported in $ \{(x,y)\in [-1,1]^2:|P(x,y)|<\delta\}$, we have
    \begin{equation}\label{eqn_2D_general_decoupling}
         \norm f_{L^p(\R^2)}\leq C_{d,\Ge}\Gd^{-\Ge}\left(\sum_{T\in \mathcal T_\Gd}\norm{\mathcal R_T f}^2_{L^p(\R^2)}\right)^{\frac 1 2},
    \end{equation}
    where $\mathcal R_T f$ denotes the Fourier restriction of $f$ to $T$.
\end{enumerate}

\end{thm}
The proof is given in Section \ref{sec_2D_uniform}.

{\it Remark.} If $P$ is of the form $y-Q(x)$, then this reduces to a previous result, namely, Theorem 1.5 of \cite{Yang2}. This explains the name ``generalised 2D uniform decoupling". Also, the theorem actually holds for all $\ell^2(L^p)$, $2\leq p\leq 6$, but we will only use the case $2\leq p\leq 4$.

\subsection{Polynomials with small Hessian determinant}
The following theorem will be used in the cylindrical decoupling part of the main proof.
\begin{thm}\label{thm_small_hessian}
For each $d\ge 2$ there is a constant $\Ga=\Ga(d)\in (0,1]$ such that the following holds. Let $\phi:\R^2\to \R$ be a polynomial of degree at most $d$, without linear terms and having $O(1)$ coefficients. If all coefficients of its Hessian determinant $\det D^2 \phi$ are bounded by some $\nu\in (0,1)$, then there exists a rotation $\rho:\R^2\to \R^2$ such that
$$
\phi\circ \rho(x,y)=A(x)+\nu^{\Ga}B(x,y),
$$
where $A,B$ are polynomials with $O(1)$ coefficients.
\end{thm}
The proof is given in Section \ref{sec_small_hessian}.

{\it Remark 1.} The exact value of $\Ga$ is not important for our purpose. In the proof given in Section \ref{sec_small_hessian}, we see that $\Ga$ can be taken to be $2^{(2-d)(d+1)^2}$, which is probably far from being sharp.

{\it Remark 2.} The rotation is necessary in general, as can be seen from the example $\phi(x,y)=x^2+(2+\nu)xy+y^2$. The restriction ``without linear terms" is also necessary, as can be seen from the example $\phi(x,y)=y+x^2$.

\subsection{Projection lemma}
We will need the following simple projection lemma. Roughly speaking, this lemma says that decoupling in higher dimensions is dominated by decoupling in one of its projections. Again, this theorem works for general Lebesgue exponents, but we will only use the $\ell^2(L^p)$ estimate.
\begin{lem}\label{lem_projection}
Let $\mathcal A$ be any finite collection of Borel subsets of $\R^2$, and denote by $D(\mathcal A)$ the smallest constant such that for all functions $g:\R^2\to \C$ Fourier supported on $\cup \mathcal A$, we have
$$
\norm g_{L^p(\R^2)}\leq D(\mathcal A)\left(\sum_{A\in \mathcal A}\norm{ \mathcal R_A g}^2_{L^p(\R^2)}\right)^{\frac 1 2}.
$$
Let $\mathcal B$ be a finite collection of Borel subsets of $\R^3$ such that there is a fixed projection $\pi:\R^3\to \R^2$ and a one-to-one correspondence between sets $A\in \mathcal A$ and $B\in \mathcal B$ via $\pi(B)=A$. Then for all functions $f:\R^3\to \C$ Fourier supported on $\cup \mathcal B$, we have
$$
\norm f_{L^p(\R^3)}\leq D(\mathcal A)\left(\sum_{B\in \mathcal B}\norm{ \mathcal R_B f}^2_{L^p(\R^3)}\right)^{\frac 1 2}.
$$
\end{lem}
The proof is elementary by freezing one variable; see for instance, Lemma 8.2 of \cite{BDG2016} and Proposition 6.2 of \cite{LiYang}.

\section{Generalised 2D uniform decoupling for polynomials}\label{sec_2D_uniform}
In this section, we prove Theorem \ref{thm_2D_general_uniform}. All implicit constants in this section are assumed to depend on the degree $d$.

\subsection{Uniform decoupling for rational functions}
Our starting point will be a uniform decoupling theorem for a large class of univariate rational functions.
\begin{defn}\label{defn_rational}
For $d\geq 0$, let $\mathcal Q(d)$ denote the collection of all rational functions $r(x)$ of the form $p_1/p_2$, where $p_1,p_2$ are univariate polynomials of degree at most $d$ and with coefficients bounded by $C_d$, and in addition $|p_2|\geq c_d$ over $[-1,1]$. Here, $C_d,c_d$ are small degree constants that will be specified later. 
\end{defn} 
We also recall the terminology from \cite{Yang2}.
\begin{defn}[Admissible partitions]
Let $I_0$ be a compact interval and $\phi:I_0\to \R$ be a smooth function. We say a partition $\mathcal P$ of $I_0$ is 
\begin{enumerate}
\item super-admissible for $\phi$ at the scale $\Gd$, if each interval in $\mathcal P$ is $(\phi,\Gd)$-flat.
\item sub-admissible for $\phi$ at the scale $\Gd$, if the union of any two adjacent intervals in $\mathcal P$ is not $(\phi,\Gd)$-flat.
\item admissible for $\phi$ at the scale $\Gd$, if it is both super-admissible and sub-admissible for $\phi$ at the scale $\Gd$.
\end{enumerate}

\end{defn}
The following proposition shows that admissible partitions generally exist.
\begin{prop}[Proposition 3.1 of \cite{Yang2}]\label{prop_3.1_Yang}
For every compact interval $I_0$, every smooth function $\phi:I_0\to \R$ and every $\Gd>0$, there exists an admissible partition $\mathcal I_\Gd$ of $I_0$ for $\phi$ at the scale $\Gd$. Moreover, for each $I\in \mathcal I_\Gd$, we also have 
\begin{equation}\label{eqn_maximally_flat}
    \sup_{x\in I}|\phi''(x)||I|^2\gtrsim \Gd.
\end{equation}
\end{prop}

We are ready to state the following uniform decoupling theorem for $\mathcal Q(d)$.
\begin{thm}\label{thm_Yang2_rational}
For any $d\geq 1$ and $\Ge>0$, there is a constant $C_{d,\Ge}$ such that the following holds. For any $0<\Gd< 1$, any rational function $r\in \mathcal Q(d)$, any sub-admissible partition $\mathcal I_\Gd$ of $[-1,1]$ for $r$ at the scale $\Gd$, any $f\in L^p(\R^2)$ Fourier supported in $\mathcal N^{r}_{O(\Gd)}([-1,1])$, we have
\begin{equation*}
    \norm f_{L^p(\R^2)}\leq  C_{d,\Ge} \Gd^{-\Ge}\left(\sum_{I\in \mathcal I_\Gd}\norm { f_I}^2_{L^p(\R^2)}\right)^{\frac 1 2}.
\end{equation*}
Here and throughout this section, for $f:\R^2\to \C$ and $I\sub \R$, $f_I$ will denote the Fourier restriction of $f$ to the strip $I\times \R$.
\end{thm}
Note that $\mathcal Q(d)$ in particular contains all polynomials of degree at most $d$ and with bounded coefficients. The corresponding weaker result has been established in Theorem 1.5 of \cite{Yang2}.
\begin{proof}
First we note that by the triangle inequality we may assume $f$ is Fourier supported in $\mathcal N^{r}_{\Gd}([-1,1])$. Then by Theorem 1.6 of \cite{Yang2} it suffices to check that $r$ is a ``polynomial-like function" of degree at most $4d$, meaning that it satisfies Definition 1.5 of \cite{Yang2}, i.e. for each $\Gs>0$ small enough and each interval $J\sub [-1,1]$, the set
\begin{equation*}
    B(r,\Gs,J):=\{x\in 
    J:|r''(x)|<\Gs (\sup_{x\in J}|r''(x)|+|J|\sup_{x\in J}|r'''(x)|)\}
\end{equation*}
is a disjoint union of at most $O_d(1)$ subintervals of $J$, and satisfies 
\begin{equation*}
    |B(r,\Gs,J)|\lesssim_d \Gs^{\frac 1 {4d}}|J|.
\end{equation*}
However, by rescaling, it suffices to check these conditions for $J=[-1,1]$. Since $r\in \mathcal Q(d)$, by direct computation we can check $r''\in \mathcal Q(4d)$ with suitably chosen $C_{4d},c_{4d}$. Thus by the fundamental theorem of algebra, $B(r,\Gs,[-1,1])$ is a disjoint union of at most $O_d(1)$ subintervals. For the length estimate, we write $r''=p_1/p_2$ with $|p_2|\gtrsim 1$ over $[-1,1]$. Thus, we have $|r''(x)|\sim |p_1(x)|$ over $[-1,1]$. On the other hand, by direct computation,
$$
\sup_{|x|\leq 1}|r'''(x)|=\sup_{|x|\leq 1}\left|\frac {p'_1p_2-p_1p'_2}{p_2^2}\right|\lesssim_d  \sup_{|x|\leq 1}|p'_1(x)|+\sup_{|x|\leq 1}|p_1(x)|\sim \sup_{|x|\leq 1}|p_1(x)|,
$$
where in the last relation we have used Proposition \ref{prop_polycoeff}. Let $p_0$ be a polynomial such that $p_0''(x) = p_1(x)$. The problem has been reduced to showing that $p_0$ is a ``polynomial-like function" of degree at most $4d$, which is implied by Lemma 1.7 of \cite{Yang2}. Thus the result follows.
\end{proof}

\subsection{The case of $O(\Gd)$-vertical neighbourhoods}
In this subsection, we let $P:\R^2\to \R$ be a polynomial of degree at most $d$ and with coefficients bounded by $1$, and assume in addition that $|P_y|\gtrsim 1$ over $[-1,1]^2$. Note that if $\Gd$ is small enough, then using the boundedness of the coefficients of $P$ we have $|P_y|\gtrsim 1$ over $[-1,1]\times [-1-C_0\Gd,1+C_0\Gd]$. Here $C_0$ is a degree constant to be specified later.

Our goal in this subsection is to prove Theorem \ref{thm_2D_general_uniform} in this special case, which further generalises Theorem \ref{thm_Yang2_rational}.

\subsubsection{A structure proposition}
Let $\Gd>0$ be small enough compared to any degree constant, and consider $\{(x,y)\in [-1,1]^2:|P(x,y)|<\Gd\}$ which is the set we will cover.
\begin{prop}\label{prop_delta_nbhd}
There is a family $\{I_k\}$ of $O(1)$ disjoint subintervals of $[-1,1]$, depending on $\Gd$, such that for each $k$, there is a unique smooth function $g_k$ satisfying $P(x,g_k(x))=0$ over $I_k$ and has its derivatives of all orders bounded. Moreover,
\begin{equation}\label{eqn_partition_k}
    \{(x,y)\in [-1,1]^2:|P(x,y)|<\Gd\}=\bigcup_{k}\{(x,y)\in I_k\times [-1,1]:|P(x,y)|<\Gd\}.
\end{equation}
In addition, for each $k$ we have the relations:
\begin{align}
    &\{(x,y)\in I_k\times [-1,1]:|P(x,y)|<\Gd\}\sub \mathcal N^{g_k}_{O(\Gd)}(I_k)\label{eqn_structure_U_delta}\\
    &\mathcal N^{g_k}_{O(\Gd)}(I_k)\sub \{(x,y)\in I_k\times[-1-O(\Gd),1+O(\Gd)]:|P(x,y)|\lesssim\Gd \}.\label{eqn_structure_U_delta_2}
\end{align}
\end{prop}

See Figure \ref{fig_structure} below for an illustration of this proposition.
\begin{figure}[h]
    \centering
    \begin{tikzpicture}[scale =1]
    %the graph
    \draw[line width=8, domain=-0.5:3.54, smooth, variable=\x, black!30] plot ({\x}, {1.5*(-\x*\x*\x*\x/4 + \x*\x + \x - \x*\x*\x/4+\x*\x*\x*\x*\x/17)});
    %the square
    \draw[line width = 1, black!50] (-0.5,-3.2) -- (3.5,-3.2) -- (3.5,.8) -- (-0.5,.8) --(-0.5,-3.2);
    
    %left graph
    \draw[dashed] (-0.5, -0.3543) --(0.397+0.135 , -0.3543) ;
    \draw[dashed] (0.397+0.135 , 1.4) -- (0.397+0.135 , -2);
    \draw[<->] (-0.5,-2) -- (0.397+0.135 , -2);
    \node[] at (-0.5/2 +0.397/2+0.135/2,-2.25) {$I_1$};
    \draw[ domain=-0.5:0.397+0.135, smooth, variable=\x, black] plot ({\x}, {1.5*(-\x*\x*\x*\x/4 + \x*\x + \x - \x*\x*\x/4+\x*\x*\x*\x*\x/17)});
    \node[] at (0,0.4) {$g_1$};
    
    %right graph
    \draw[dashed] (3.5 , -2.383)--(2.599-0.135 , -2.383)  ;
    \draw[dashed] (2.599-0.135 , -2.7)-- (2.599-0.135 , 1.8);
    \draw[<->] (2.599-0.135 , -2.7) -- (3.5 , -2.7);
    \node[] at (2.599/2-0.135/2 + 3.5/2,-2.95) {$I_2$};
    \draw[ domain=2.599-0.135:3.5, smooth, variable=\x, black] plot ({\x}, {1.5*(-\x*\x*\x*\x/4 + \x*\x + \x - \x*\x*\x/4+\x*\x*\x*\x*\x/17)});
    \node[] at (3.1,-0.2) {$g_2$};
    
    %O(\delta)
    \draw[dashed] (3.5,.8) -- (3.5,1.8)--(-0.5,1.8);
    \draw[<->] (-0.6,-0.5) -- (-0.6,-0.2);
    \node[] at (-1,-0.35) {\tiny $O(\delta)$};
    
    \draw[<->] (3.6,.8) -- (3.6,1.8);
    \node[] at (4,1.3) {\tiny $O(\delta)$};

    %legend
    \draw[line width=8,black!30] (4,-1) -- (5,-1);
    \node[] at (5.75,-1) {\tiny $:\{|P|<\delta\}$};
    
    \end{tikzpicture}
    \caption{Illustration of the structure proposition}
    \label{fig_structure}
\end{figure}
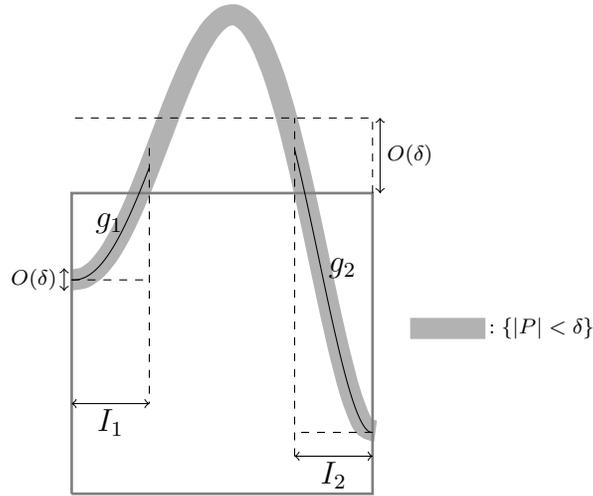
\begin{proof}
We may assume $P_y\gtrsim 1$ over $[-1,1]\times [-1-C_0\Gd,1+C_0\Gd]$.

By the continuity of $P$, let $I_k \subset [-1,1]$ be disjoint closed intervals whose union is 
$$
\{ x\in [-1,1]:\exists \,\,y_0 \in [-1,1] \,\,\text{such that} \,\, |P(x,y_0)|\leq\Gd\}.
$$
The endpoints of these intervals satisfy one of the 4 equalities: $x=-1$, $x=1$, $P(x,1)=-\Gd$ or $P(x,-1)=\Gd$. If either of the latter two polynomial equations holds for all $|x|\le 1$, then $[-1,1]^2\cap\{|P|<\Gd\}=\varnothing$ and so we have nothing to prove. Otherwise, by the fundamental theorem of algebra, there are at most $O(1)$ many solutions to above equations, and thus the number of these intervals $I_k$ is $O(1)$. Also, \eqref{eqn_partition_k} holds automatically.

Now fix $k$. For each $x\in I_k$, the function $y\mapsto P(x,y)$ is strictly increasing with derivative bounded away from $0$. We now show that if $C_0$ is chosen properly, then there is a unique $g_k(x)\in [-1-C_0\Gd,1+C_0\Gd]$ such that $P(x,g_k(x))=0$. Indeed, if $-\Gd<P(x,y_0)<0$, then we compute
$$
P(x,1+C_0\Gd)-P(x,y_0)=\int_{y_0}^{1+C_0\Gd} P_y(x,t)dt\gtrsim C_0\Gd.
$$
For $C_0$ large enough we have $P(x,1+C_0\Gd)>0$. Thus by the intermediate value theorem, there exists a unique $g_k(x)\in [y_0,1+C_0\Gd]$ such that $P(x,g(x))=0$. The case $0\leq P(x,y_0)<\Gd$ is similar. The smoothness of $g_k$ and the boundedness of $g_k'$ follow from the implicit function theorem and $P_y\gtrsim 1$.

Moreover, by the mean value theorem we have
$$
\Gd> |P(x,g_k(x))-P(x,y_0)|=|P_y(x,z)(g_k(x)-y_0)|,
$$
where $z\in [-1-C_0\Gd,1+C_0\Gd]$ and thus $|P_y(x,z)|\gtrsim 1$. Thus $|g_k(x)-y_0|\lesssim \Gd$. This in particular shows that $|g_k(x)|-1\lesssim\Gd$. A similar argument shows that for any $x\in I_k$, $|y|\le 1$ with $|P(x,y)|<\Gd$, we have $|y-g_k(x)|\lesssim \Gd$. This proves \eqref{eqn_structure_U_delta}. On the other hand, if $|y-g_k(x)|\lesssim \Gd$, then $|y|-1\lesssim \Gd$ and
$$
|P(x,y)|=|P(x,y)-P(x,g_k(x))|\lesssim |y-g_k(x)|\lesssim \Gd,
$$
since $P$ has bounded coefficients. This proves \eqref{eqn_structure_U_delta_2}.
\end{proof}

{\bf A reduction.} In view of this proposition and using triangle and H\"older's inequalities, in the following we may assume without loss of generality that there is only one $I=I_k=[-1,1]$, and the implicit function $g=g_k$ is defined throughout $[-1,1]$. Also, by \eqref{eqn_structure_U_delta} and \eqref{eqn_structure_U_delta_2}, to decouple $[-1,1]^2\cap \{|P|<\Gd\}$, it suffices to decouple $\mathcal N^{g}_{O(\Gd)}([-1,1])$.

\subsubsection{Main decoupling argument}

\begin{prop}\label{prop_P_y_large}
Let $\Gd>0$ and $g:[-1,1]\to [-1-C_0\Gd,1+C_0\Gd]$ be as in Proposition \ref{prop_delta_nbhd} (with the reduction right above). Then there is a partition $\mathcal I_\Gd$ of $[-1,1]$ into $(g,\Gd)$-flat intervals $I$ such that for any $f\in L^p(\R^2)$ Fourier supported in $\mathcal N^g_{O(\Gd)}([-1,1])$, we have
\begin{equation}\label{eqn_decoupling_P_y_large}
    \norm f_{L^p(\R^2)}\lesssim_\Ge \Gd^{-\Ge}\left(\sum_{I\in \mathcal I_\Gd}\norm { f_I}^2_{L^p(\R^2)}\right)^{\frac 1 2}.
\end{equation}
Also, for each absolute constant $C\geq 1$, the intervals $\{CI:I\in \mathcal I_\Gd\}$ have bounded overlap.
\end{prop}

Proposition \ref{prop_P_y_large} is formulated to facilitate an induction on scales argument used to prove it. However, as in the proof of Theorem \ref{thm_2D_general_uniform}, we shall use the following corollary instead.

\begin{cor}\label{cor_P_y_large}
Let $P:\R^2\to \R$ be a polynomial of degree at most $d$, with $O(1)$ coefficients and satisfying $|P_y|\gtrsim 1$ over $[-1,1]^2$. Then for $\Gd>0$, there is a cover $\mathcal T_\Gd$ of $[-1,1]^2\cap \{|P|<\Gd\}$ by rectangles $T$, such that for each absolute constant $C\geq 1$, the rectangles $CT$ have bounded overlap and on each $CT$ we have $|P|\lesssim \Gd$. In addition, we have the following decoupling inequality:
for any $f\in L^p(\R^2)$ Fourier supported on $[-1,1]^2\cap \{|P|<\Gd\}$, we have
\begin{equation}\label{eqn_decoupling_P_y_large_cor}
    \norm f_{L^p(\R^2)}\lesssim_\Ge \Gd^{-\Ge}\left(\sum_{T\in \mathcal T_\Gd}\norm { \mathcal R_T f}^2_{L^p(\R^2)}\right)^{\frac 1 2}.
\end{equation}
\end{cor}

\begin{proof}[Proof of Corollary \ref{cor_P_y_large} assuming Proposition \ref{prop_P_y_large}]
As mentioned after Proposition \ref{prop_delta_nbhd}, it suffices to decouple the vertical neighbourhood $\mathcal N^g_{O(\Gd)}([-1,1])$. Proposition \ref{prop_P_y_large} gives a partition $\mathcal I_\Gd$ of $[-1,1]$ into subintervals $I$. Since $g$ is $O(\Gd)$-flat over each $I$ and has bounded derivatives, each $\mathcal N^g_{O(\Gd)}(I)$ can be extended to an equivalent rectangle $T$. Also, the proof of Lemma \ref{lem_rect=para} shows that the projection of $T$ into the $x$ axis is contained in $C'I$ for some absolute constant $C'$, and thus the projection of $CT$ into the $x$ axis is contained in $CC'I$. Since we established in Proposition \ref{prop_P_y_large} the bounded overlap of $CC'I$, the bounded overlap of the rectangles $CT$ then follows.

By the second relation \eqref{eqn_structure_U_delta_2} of Proposition \ref{prop_delta_nbhd}, and Corollary \ref{cor_polycoeff}, we have $|P|\lesssim \Gd$ over $CT$. The decoupling inequality \eqref{eqn_decoupling_P_y_large_cor} follows immediately from \eqref{eqn_decoupling_P_y_large}, the bounded overlap of $\mathcal T_\Gd$ and a standard tiling argument followed by triangle and H\"older's inequalities.
\end{proof}

\subsubsection{Proof of Proposition \ref{prop_P_y_large}}
\begin{proof}
The main idea is a Pramanik-Seeger type iteration, namely, we induct on scales, locally approximate $g$ by rational functions, and apply Theorem \ref{thm_Yang2_rational} in each iteration.

We first consider the base case $c<\Gd<1$ for a small enough degree constant $c$ to be determined. In this case the result holds trivially by taking $\mathcal I_\Gd=\{[-1,1]\}$. 

Now given $0<\Gd<c$. We argue through the following steps.

\underline{Step 1: Construction of partition.}
We apply Proposition \ref{prop_3.1_Yang} to find an admissible partition $\mathcal I_\Gd$ of $[-1,1]$ for $g$ at the scale $\Gd$. For each absolute constant $C\geq 1$, the bounded overlap of intervals $\{CI:I\in \mathcal I_\Gd\}$ is proved in Proposition \ref{prop_bounded_overlap} below. We postpone the proof since it is lengthy and we do not want to overwhelm the reader by too many technicalities here.

\underline{Step 2: Choosing intermediate scales.}
To prove the decoupling inequality, our goal is to approximate $g$ by rational functions at each different scale. First, let us denote by $N$ the smallest integer such that 
$$
\Gd^{(2/3)^N}:=\Gd_0\geq c.
$$
Denote $\Gd_n=\Gd_0^{(3/2)^n}$, and so $\Gd_N=\Gd$ and $\Gd_n=\Gd_{n-1}^{3/2}$ for each $n$. (We remark that a more intuitive choice is such that $\Gd_n=\Gd_{n-1}^2$, but due to some slight technicality it is easier to work with $\Gd_n=\Gd_{n-1}^{2-\Ge_0}$ for any fixed small $\Ge_0$. We take $\Ge_0=1/2$.) 

For each $1\leq n\leq N-1$, we also apply Proposition \ref{prop_3.1_Yang} to find an admissible partition $\mathcal I_n$ of $[-1,1]$ for $g$ at the scale $\Gd_n$. With this, let use denote by $D_n$ the smallest constant such that for all $f$ Fourier supported in $\mathcal N^g_{O(\Gd)}([-1,1])$, we have
\begin{equation}\label{eqn_defn_decoupling_constant}
    \norm f_{L^p(\R^2)}\leq D_n\left(\sum_{I\in \mathcal I_{n}}\norm { f_{I}}^2_{L^p(\R^2)}\right)^{\frac 1 2}.
\end{equation}
Our goal is to show that $D_N\lesssim_\Ge \Gd^{-\Ge}$ for every $\Ge>0$.

\underline{Step 3: Technical treatment of overlap.}
We hope that the partitions $\mathcal I_n$ are nested in the sense that $\mathcal I_n$ is finer than $\mathcal I_{n-1}$ for each $n$. Below we shall prove that without loss of generality this is indeed the case. 

We only show the case $n=N$, as the other cases are the same. Let $\mathcal I_{N-1}=\{I_j\}$ arranged from left to right. Consider $I_1$ and denote by $p$ its right endpoint. It may happen that $p$ coincides with a boundary point of some $I\in\mathcal I_N$, or lies in the interior of some $I\in \mathcal I_N$. The previous case is good. In the latter case, since $\Gd_{N-1}>\Gd_N$, by admissibility such an $I$ will contain at most one such $p$. Let $\tilde I_{1}:=I_{1}\cup I$, which is equal to a union of intervals in $\mathcal I_N$. Since the rough cutoff Fourier multiplier $(x,y)\mapsto 1_{I_1}(x)$ is bounded by an absolute constant $C$ on $L^p(\R^2)$, we have
$$
\norm{f_{I_1}}_{L^p(\R^2)}\leq C\norm {f_{\tilde I_1}}_{L^p(\R^2)}.
$$
For other intervals $I_j$, we potentially need two intervals $I$ and $I'$ to the left and right of $I_j$ respectively. A similar argument shows that $\norm{f_{I_j}}_4\leq C\norm {f_{\tilde I_j}}_4$ where $\tilde I_{j}:=I_{j}\cup I \cup I'$. 

By the argument above, to decouple $[-1,1]$ into intervals $I\in \mathcal I_N$, it suffices to complete two tasks: decoupling $[-1,1]$ into each $I_{j}\in \mathcal I_{N-1}$, and decoupling each $\tilde I_{j}$ into intervals $\{I\in \mathcal I_N,I\sub \tilde I_{j}\}$. Therefore, we may assume without loss of generality that each $\mathcal I_n$ is a refinement of $\mathcal I_{n-1}$, at the cost of losing an absolute constant in the decoupling inequality of each iteration.

\underline{Step 4: Approximation.}
Fix $I_{N-1}\in \mathcal I_{N-1}$, which is a union of intervals in $\mathcal I_N$. We will decouple $I_{N-1}$ into intervals $\{I\in \mathcal I_N:I\sub I_{N-1}\}$.

To this end, noting that $I_{N-1}$ is $(g,\Gd_{N-1})$-flat, we have
$$
\sup_{x\in I_{N-1}}|g(x)-g(x_0)-g'(x_0)(x-x_0)|\leq \Gd_{N-1},
$$
where $x_0$ is the left endpoint of $I_{N-1}$. We may assume without loss of generality that $g(x_0)=g'(x_0)=0$ by a shear transformation plus a translation. The main reason is that 
$$
Q(x,y):=P(x,y+g(x_0)+g'(x_0)(x-x_0))
$$ 
is also a polynomial of degree at most $d$ and with bounded coefficients (since $g,g'$ are bounded), and satisfies $|Q_y|\gtrsim 1$. Thus, we have $|g|\leq \Gd_{N-1}$ over $I_{N-1}$.

Write
$$
P(x,y)=A(x)+yB(x)+y^2C(x,y)
$$
for polynomials $A,B,C$ with bounded coefficients. Using the assumption that $|P_y|\gtrsim 1$ over $[-2,2]^2$, we have $|B(x)|\gtrsim 1$ over $[-2,2]$. 

We will approximate $g$ by the rational function $r(x)=r_{I_N}(x)$ satisfying
\begin{equation}\label{eqn_defn_approximation}
  A(x)+r(x)B(x)=0,  
\end{equation}
that is,
$$
r(x)=-A(x)/B(x).
$$
Note that $r\in \mathcal Q(d)$ as in Definition \ref{defn_rational} if $C_d,c_d$ are suitably chosen. 

Recall by definition $P(x,g(x))=0$, or
\begin{equation*}
    A(x)+g(x)B(x)+g(x)^2 C(x,g(x))=0.
\end{equation*}
Using \eqref{eqn_defn_approximation}, we have
\begin{equation}\label{eqn_approximation_2}
    (g(x)-r(x))B(x)+g(x)^2 C(x,g(x))=0.
\end{equation}
Using $|g(x)|\leq \Gd_{N-1}$ and $|B(x)|\gtrsim 1$, this gives $|g(x)-r(x)|\lesssim \Gd_{N-1}^2=\Gd_N^{4/3}$. Since $\Gd_N=\Gd<c$, if $c$ is chosen to be small enough, then we have $|g(x)-r(x)|<\Gd/2$. Thus, the graph of the function $g$ over $I_{N-1}$ is contained in $\mathcal N^r_{\Gd/2}(I_{N-1})$.

\underline{Step 5: Proving sub-admissibility for $r$.}
We now recall that $\mathcal I_{N}=\mathcal I_\Gd$ is an admissible partition of $[-1,1]$ for $g$ at the scale $\Gd_{N}=\Gd$. In particular, the partition $\mathcal I_{N}(I_{N-1})$ consisting of all members of $\mathcal I_{N}$ that are subintervals of $I_{N-1}$ is an  admissible partition of $I_{N-1}$ for $g$ at the scale $\Gd$. To use Theorem \ref{thm_Yang2_rational}, we need to show that $\mathcal I_{N}(I_{N-1})$ is sub-admissible for $r$ at the scale $\Gd/2$. Suppose not. Then there are two adjacent intervals $J_1,J_2$ in $\mathcal I_{N}(I_{N-1})$ such that $r$ is $\Gd/2$-flat over $J_1\cup J_2$.

Differentiating \eqref{eqn_approximation_2} gives
$$
(g'(x)-r'(x))B(x)+(g(x)-r(x))B'(x)+2g(x)g'(x)C(x,g(x))+g(x)^2 (C(x,g(x)))'=0.
$$
Using $g=O(\Gd_{N-1})$ over $I_{N-1}$ and Part \eqref{item_linear2} of Lemma \ref{lem_appendix_flat}, we have $g'|I_{N-1}|=O(\Gd_{N-1})$. Thus, using $|B|\gtrsim 1$ we also have
\begin{equation}\label{eqn_g'-r'}
    \sup_{x\in I_{N-1}}|g'(x)-r'(x)||I_{N-1}|\lesssim \Gd_{N-1}^2=\Gd^{4/3}.
\end{equation}
Since $|r-g|=O(\Gd^{4/3})$ and we assume $r$ is $\Gd/2$-flat over $J_1\cup J_2\sub I_{N-1}$, by the triangle inequality we have
$$
\sup_{x,x_0\in J_1\cup J_2}|g(x)-g(x_0)-g'(x_0)(x-x_0)|\leq \Gd/2+O(\Gd^{4/3})<\Gd,
$$
and hence each $J_1\cup J_2$ is also $(g,\Gd)$-flat. This is a contradiction since we chose  $\mathcal I_\Gd$ to be admissible (and hence sub-admissible) for $g$ at the scale $\Gd$. As a result, $\mathcal I_{N}(I_{N-1})$ is sub-admissible for $r$ at the scale $\Gd/2$.

\underline{Step 6: Applying Theorem \ref{thm_Yang2_rational}.}
Now it is time we use Theorem \ref{thm_Yang2_rational} to decouple $I_{N-1}$ into the family $\mathcal I_{N}(I_{N-1})$. More precisely, for every function $f$ Fourier supported in $\mathcal N^r_{O(\Gd)}(I_{N-1})$, we have
\begin{equation}\label{eqn_Yang2_rational}
    \norm f_{L^p(\R^2)}\leq C_{d,\Ge}\Gd^{-\Ge}\left(\sum_{I\in \mathcal I_{\Gd}(I_{N-1})}\norm { f_{I}}^2_{L^p(\R^2)}\right)^{\frac 1 2}.
\end{equation}
Since $|g-r|\leq \Gd/2$, the above decoupling inequality will hold for every $f$ Fourier supported on $\mathcal N^g_{O(\Gd)}(I_{N-1})$.

\underline{Step 7: Induction on scales.}
In this way, we have completed the decoupling of each $I_{N-1}\in \mathcal I_{N-1}$ into intervals $\mathcal I_N$. Thus, it remains to decouple $[-1,1]$ into intervals $\mathcal I_{N-1}$. Therefore, the above argument implies the following
bootstrap inequality: for every $\Ge>0$,
$$
D_N\leq C_{d,\Ge}\Gd_N^{-\Ge}D_{N-1}.
$$
Iterating this inequality for $N=O(\log\log \Gd^{-1})$ times until we arrive at $D_0\sim 1$, we have
$$
D_N\leq (\log \Gd^{-1})^{\log C_{d,\Ge}} \Gd^{-3\Ge}D_0\lesssim_\Ge \Gd^{-4\Ge}.
$$
Since $\Ge>0$ is arbitrary, we are done.

\end{proof}

\subsection{Proof of generalised 2D uniform decoupling for polynomials}\label{sub_proof_gen_2D_uniform}
We are now ready to fully prove Theorem \ref{thm_2D_general_uniform}.
\begin{proof}
The proof is by induction on the degree $d$.  For $d=0$ this is trivial. Assuming the result holds for $d-1\geq 0$, we now prove it for degree $d$, through the following steps.

\subsubsection{Dyadic decomposition}
Let $c<1,C>1$ be two degree constants to be determined. With this, we dyadically decompose $[-1,1]^2$ into the following subsets
\begin{align*}
    V_0 &:=\{(x,y)\in [-1,1]^2:|P_x(x,y)|<C\Gd,|P_y(x,y)|<C\Gd\}\\
    V_{\Gs_1,0}&:=\{(x,y)\in [-1,1]^2:C\Gs_1\le |P_x(x,y)|<C(1+2c)\Gs_1,|P_y(x,y)|<\Gd\}\\
    V_{0,\Gs_2}&:=\{(x,y)\in [-1,1]^2:|P_x(x,y)|<\Gd,C\Gs_2\le |P_y(x,y)|<C(1+2c)\Gs_2\}\\
    V_{\Gs_1,\Gs_2}&:=\{(x,y)\in [-1,1]^2:\Gs_1\le |P_x(x,y)|<(1+2c)\Gs_1,\Gs_2\le |P_y(x,y)|<(1+2c)\Gs_2\}
\end{align*}
for $\Gd\leq\Gs_i\lesssim 1$, $i=1,2$. Since we can tolerate logarithmic losses, it suffices to cover each of the above types of subsets (which we will refer to as a {\it dyadic layer} in this proof).

\subsubsection{Decoupling $V_0\cap \{|P|<\Gd\}$}\label{subsub_V_delta}
Since $P_y$ has degree at most $d-1$ and coefficients bounded by $d$, we can apply the induction hypothesis to $P_y$ to decouple $\{|P_y|<C\Gd\}$ into boundedly overlapping rectangles $T'$, on each of which we have $|P_y|\lesssim \Gd$. The remaining steps are as follows.

\underline{Step A1: Rotation and translation.}
Fix a rectangle $T'$. Let $\rho$ be a rotation by no more than $\pi/4$ such that $\rho(T')$ is axis parallel. We still have $|(P\circ \rho)_y|\lesssim \Gd$ since $|P_x|\lesssim \Gd$. For this reason, in the following we may assume without loss of generality that $T'$ is axis-parallel. Thus, we can write $T'=I'\times J'$ where $I',J'$ are subintervals of $[-1,1]$.

\underline{Step A2: Approximation and cylindrical decomposition.}
We now approximate $P(x,y)$ by $P(x,0)$ which depends on $x$ only. Indeed, by the mean value theorem, on $T'$ we have
\begin{equation}\label{eqn_approx_V_delta}
   |P(x,y)-P(x,0)|\leq C|y|\Gd\leq C\Gd. 
\end{equation}
Now fix a point $(x,y)\in T'\cap \{|P|<\Gd\}$, so $|P(x,y)|\leq \Gd$ and $|P(x,0)|\leq 2C\Gd$ by \eqref{eqn_approx_V_delta}. By the fundamental theorem of algebra, the set $\{x\in I':|P(x,0)|<2C\Gd\}$ is a disjoint union of $O(1)$ intervals $I$. In this way, we have decoupled $T'\cap \{|P|<\Gd\}$ into rectangles $T=I\times J'$. 

Now we check that the rectangle $T$ is as required, namely, on $T$ we have $|P|\lesssim \Gd$. Indeed, for $(x,y)\in T$ we have $|P(x,0)|\leq 2C\Gd$, and since $T\sub T'$, by \eqref{eqn_approx_V_delta} we have $|P(x,y)|\leq 3C\Gd$. Thus we are done.

\subsubsection{Decoupling $V_{\Gs_1,\Gs_2}\cap \{|P|<\Gd\}$}\label{subsub_V_sigma}
On $V_{\Gs_1,\Gs_2}\cap \{|P|<\Gd\}$ we have
$$
|P_x(x,y)+ (1+c)\Gs_1|<c\Gs_1\quad \text{or}\quad |P_x(x,y)- (1+c)\Gs_1|<c\Gs_1.
$$
Replacing $x$ by $-x$ if necessary, it suffices to prove the latter case.

\underline{Step B1: Applying induction hypothesis.}
Since $P_x$ has degree at most $d-1$ and coefficients bounded by $d$, we may apply the induction hypothesis to $P_x-(1+c)\Gs_1$ to decouple $\{|P_x-(1+c)\Gs_1|<c\Gs_1\}$ into rectangles $T'_1$, such that $|P_x-(1+c)\Gs_1|\lesssim c\Gs_1$ over $100T'_1$. Thus, if the degree constant $c$ in the definition of $V_{\Gs_1,\Gs_2}$ is chosen to be small enough, then on $100T'_1$ we have $P_x\sim \Gs_1$. Also, note that by the induction hypothesis, the rectangles $100T'_1$ have bounded overlap.

Similarly, we may apply the same argument to $P_y$ to decouple $\{|P_y\pm(1+c)\Gs_2|<c\Gs_2\}$ into rectangles $T'_2$, such that $|P_y|\sim \Gs_2$ over $100T'_2$, and the rectangles $100T'_2$ have bounded overlap. It suffices to prove the cases ``$+$" and ``$-$" separately, and replacing $y$ by $-y$ if necessary, it suffices to prove the ``$-$" case, in which $P_y\sim \Gs_2$ over $100T'_2$.

Now by Lemma \ref{prop_rectangle_intersection}, there is a rectangle $T'$ such that 
$$
T'_1\cap T'_2\sub  T'\sub 100 T'_1\cap 100 T'_2,
$$
on which we have both $P_x\sim \Gs_1$ and $P_y\sim \Gs_2$. By symmetry, it further suffices to assume $\Gs_1\leq \Gs_2$. Also, the rectangles $T'$ have bounded overlap.

\underline{Step B2: Rotation and translation.}
We will need the following lemma, which studies the interaction between partial derivative bounds and rotations.
\begin{lem}\label{lem_partial}
Let $T'$ and $P$ be as in Step B1, and let $|\theta|\leq \pi/4$ and $\rho$ be the counterclockwise rotation by $\theta$. Denote $Q=P\circ \rho$. Then we have
\begin{enumerate}
    \item We have $\max\{Q_x,Q_y\}\lesssim \Gs_2$ over $\rho^{-1}(T')$.
    \item Either $Q_x\sim \Gs_2$ over $\rho^{-1}(T')$, or $Q_y\sim \Gs_2$ over $\rho^{-1}(T')$.
\end{enumerate}

\end{lem}
\begin{proof}
By direct computation, 
$$
Q_x=(P_x\circ \rho) \cos\theta+(P_y\circ \rho) \sin \theta, \quad Q_y=-(P_x\circ \rho) \sin\theta+(P_y\circ \rho) \cos \theta.
$$
The upper bound $\max\{Q_x,Q_y\}\lesssim \Gs_2$ is trivial. For the lower bound of $Q_y$, we have two cases.
\begin{itemize}
    \item If $|\theta|\leq \pi/6$, then $|Q_y|\gtrsim \frac{\sqrt 3}2 \Gs_2-\frac 1 2\Gs_1\gtrsim \Gs_2$.
    \item If $\pi/6<|\theta|\leq \pi/4$, then
    \begin{itemize}
        \item If $\theta\geq 0$, then $Q_x\geq P_y/2\sim \Gs_2$.
        \item If $\theta<0$, then $Q_y\geq P_y/\sqrt 2\sim \Gs_2$.
    \end{itemize}
\end{itemize}
\end{proof}

We now proceed in similar ways as when we decouple $V_0\cap \{|P|<\Gd\}$ in Section \ref{subsub_V_delta}. First, similar to Step A1, we may apply a rotation by no more than $\pi/4$ to make $T'$ axis parallel. Using Lemma \ref{lem_partial}, after the rotation we have $Q_x\lesssim Q_y\sim \Gs_2$, or $Q_y\lesssim Q_x\sim \Gs_2$. By symmetry, we only prove the former case. Also, by abuse of notation we may assume a priori that $T'=I'\times J'$ where $I',J'$ are subintervals of $[-1,1]$, and on $T'$ we have $P_x\lesssim P_y\sim \Gs_2$. Write $J'=[-\eta,\eta]$ for future reference.

\underline{Step B3: Approximation and cylindrical decomposition.}
Next, similar to Step A2, we approximate $P(x,y)$ by $P(x,0)$ with error $O(\eta\Gs_2)$. We have two cases. 

If $\eta\Gs_2\leq \Gd$, then $P(x,y)$ can be approximated by $P(x,0)$ with an error of $O(\Gd)$. Thus, the result follows from the same proof in the end of Step A2.

If $\eta\Gs_2>\Gd$, we fix a point $(x,y)\in \{|P|<\Gd\}\cap T'$, so $|P(x,y)|<\Gd$ and thus $|P(x,0)|\lesssim \eta\Gs_2$. By the fundamental theorem of algebra, the set $\{x\in I':|P(x,0)|\lesssim \eta\Gs_2\}$ is a disjoint union of $O(1)$ intervals $I''$. Thus, we have decomposed $T'$ into rectangles $T''=I''\times [-\eta,\eta]$. Also, by the approximation we have $|P|\lesssim \eta\Gs_2$ over $T''$.

\underline{Step B4: Rescaling.}
We need to decouple $T''$ further. To do this, we do a rescaling. Let $l$ be a linear map that maps $[-1,1]$ to $I''$, and consider the polynomial $P(lx,\eta y)$, which is essentially bounded by $\eta\Gs_2$ over $[-1,1]^2$, and hence all its coefficients are essentially bounded by $\eta\Gs_2$ by Proposition \ref{prop_polycoeff}. Define
$$
\tilde P(x,y)= (\eta\Gs_2)^{-1}P(lx,\eta y)
$$
so that $\tilde P$ has bounded coefficients. Now recall that $P_y\sim \Gs_2$ on $T'$, and so $P_y\sim \Gs_2$ on $T''$ since $T''\sub T'$. Thus $\tilde P_y\sim 1$ over $[-1,1]^2$.

\underline{Step B5: Applying Corollary \ref{cor_P_y_large}.}
Now we apply Corollary \ref{cor_P_y_large} to $\tilde P$ to decouple $\{|\tilde P|< (\eta\Gs_2)^{-1}\Gd\}$ into boundedly overlapping rectangles, on each of which we have $|\tilde P|\lesssim (\eta\Gs_2)^{-1}\Gd$. Reversing the rescaling, we obtain a boundedly overlapping cover of $\{|P|<\Gd\}$ by rectangles on each of which $|P|\lesssim \Gd$. 

\subsubsection{Decoupling $V_{\Gs_1,0}$ and $V_{0,\Gs_2}$}
By symmetry, we only mention how to decouple $V_{0,\Gs_2}$. Apply the induction hypothesis in a way similar to Step B1 to obtain rectangles $T'$ on which $P_x\lesssim \Gd$ and $P_y\sim C\Gs_2$. Since $\Gs_2\geq \Gd$, if $C$ is initially chosen to be large enough, then on $T'$ we have $P_y\geq 2P_x$. Thus, similar to Step B2, we perform a rotation $\rho$ by no more than $\pi/4$ so that $T'$ becomes axis parallel, and we have $(P\circ \rho)_y\gtrsim \Gs_2$. The remaining steps are the same as Steps B3, B4 and B5.

\subsubsection{Bounded overlap}
Lastly, we need to show the bounded overlap of $\mathcal T_\Gd$. First, by Corollary \ref{cor_P_y_large}, the family of rectangles $T$ covering $V_0$ obeys that $100T$'s have bounded overlap. The same holds true for rectangles covering a single dyadic layer of each of the other 3 types, namely $V_{\Gs_1,0},V_{0,\Gs_2},V_{\Gs_1,\Gs_2}$.

We still need to show the bounded overlap among the enlarged rectangles covering different dyadic layers. We first consider two different dyadic layers of the form $V_{\Gs_1,\Gs_2}$, $V_{\Gs'_1,\Gs'_2}$. If $T$ is a rectangle in the final cover $\mathcal T_\Gd$ of $V_{\Gs_1,\Gs_2}$, then the proof shows that on $100T$ we have $|P_x|\sim \Gs_1$, $|P_y|\sim \Gs_2$. Similarly, if $S$ is a rectangle in the final cover $\mathcal T_\Gd$ of some other $V_{\Gs'_1,\Gs'_2}$, then on $100S$ we have $|P_x|\sim \Gs'_1$, $|P_y|\sim \Gs'_2$. If $\Gs_1\not\sim \Gs_2$ or $\Gs_2\not\sim\Gs'_2$, then $100T\cap 100S=\varnothing$. This ensures that the 100-enlarged rectangles covering different dyadic layers of the form $V_{\Gs_1,\Gs_2}$ have bounded overlap. In similar ways, we can show the bounded overlap among the 100-enlarged rectangles covering different dyadic layers of all four types $V_0$, $V_{\Gs_1,0}$, $V_{0,\Gs_2}$ and $V_{\Gs_1,\Gs_2}$. Hence the result follows.

%Lastly, for the decoupling inequality, by triangle and H\"older's inequalities and the bounded overlap we have just established, each step of the iteration gives rise to a factor of at most $C_{d,\Ge}\Gd^{-\Ge}$ to the decoupling constant. Since we iterate at most $d$ times, the final decoupling inequality holds with the decoupling constant of the order $C_{d,\Ge}\Gd^{-\Ge}$.

\end{proof}

\subsection{Technicalities}
In this subsection, we deal with some purely technical issues that were left behind. We choose to include them here to avoid distracting the reader from the key concepts in our main proof.
\subsubsection{Intersection of rectangles}
Rectangles are a main geometric object in decoupling theory and many other related topics such as restriction and Kakeya problems. However, the intersection of rectangles is not guaranteed to be a rectangle, and this gives rise to some minor technicalities, which we discuss below.

\begin{lem}\label{lem_rect=para}
For every parallelogram $P\sub \R^2$, there is a rectangle $T$ such that $P\sub T\sub 3P$.
\end{lem}

\begin{proof}
If $P$ is a rectangle then we have nothing to show. If not, then by rotation and rescaling we may represent $P$ by the inequalities 
$$
|x|\leq 1, \quad |y-mx|\leq b
$$
where $0<m<\infty$, $b>0$. If $m\leq b$, then we extend $P$ to a rectangle $T$ along the sides $x=\pm 1$. If $m>b$, then we extend $P$ to a rectangle $T$ along the sides with slope $m$. In either case, the ratio of the extended part is at most $1$. The situations are illustrated in Figures \ref{fig_para=rect} and \ref{fig_para=rect2} below. We leave the details of the computation to the reader.

\end{proof}

\begin{figure}[h]
    \begin{minipage}[t]{0.45\textwidth}
    \centering
    \begin{tikzpicture}
    %The case m \leq b
    \draw (0,0) -- (0,3) -- (6,4) -- (6,1) -- (0,0) ;
    \draw (2,4/3) -- (2,7/3)  -- (4,8/3) -- (4,5/3)-- (2,4/3);
    \draw[dashed] (4,8/3) -- (2,8/3) -- (2,4/3) -- (4,4/3) -- (4,8/3);
    %tags
    \node[] at (3.2,2.9) { $T$};
    \node[] at (4.2,2.2) {$P$};
    \node[] at (4.2,0.3) {$3P$};
    \node[] at (0,-.5) {};
    \end{tikzpicture}
    \caption{The case $m \leq b$}
    \label{fig_para=rect}
    \end{minipage}
    \hfill
    \begin{minipage}[t]{0.45\textwidth}
    \centering
    \begin{tikzpicture}
    %The case m \geq b
    \draw (0,0) -- (3,1) -- (3,6) -- (0,5)--(0,0) ;
    \draw (1,2) -- (2,7/3) -- (2,4) -- (1,11/3) -- (1,2);
    \draw[dashed] (1,2) -- (2.5,2.5) -- (2,4)--(.5,3.5) -- (1,2);
    %tags
    \node[] at (2.4,3.5) { $T$};
    \node[] at (1.5,1.9) {$P$};
    \node[] at (3.4,3) {$3P$};
    \end{tikzpicture}
    \caption{The case $m \geq b$}
    \label{fig_para=rect2}
    \end{minipage}
\end{figure}

The following proposition is used in Section \ref{sub_proof_gen_2D_uniform}.
\begin{lem}\label{prop_rectangle_intersection}
Let $T_1,T_2\sub \R^2$ be rectangles with nonempty intersection. Then there is a rectangle $T$ such that 
%https://link.springer.com/content/pdf/10.1007/s11590-015-0941-0.pdf?fbclid=IwAR3M0MSavlwW4sLR9e3i7tU1KQERSYuG_7RVHrDmmmHfXqnMQq2TamRc_EU}
$$
T_1\cap T_2\sub T\sub 100T_1\cap 100T_2.
$$
\end{lem}
\begin{proof}
Firstly, by rescaling, translation and Lemma \ref{lem_rect=para} it suffices to prove the case when $T_1=[-1,1]^2$. Here we use Lemma \ref{lem_rect=para} twice and so we lose a factor of at most $9$. We also used the following simple facts. First, for any invertible affine transformation $L:\R^2\to \R^2$, any parallelogram $P$ and any $C\geq 1$, we have $L(CP)=CL(P)$. Second, using convexity, it is direct to show that if $P_1\sub P_2$ are parallelograms and $C\geq 1$, then $CP_1\sub CP_2$ (the dilations are with respect to different centres) . We leave the details of the reduction to the reader.

Now we come to the main proof. Let $T_2$ have side lengths $l_1\le l_2$. We distinguish the following cases.
\begin{enumerate}
    \item If $l_1\le l_2\le 2$, then using $T_1\cap T_2\neq\varnothing$ we see that $T_2\sub (1+2\sqrt 2)T_1$.
    \item If $2\le l_1\le l_2$, then similarly we have $T_1\sub (1+2\sqrt 2)T_2$.
    \item If $l_1\le 2\le l_2$, we consider a tiling of $T_2$ by squares $Q$ of side length $l_1$ (with the same orientation as the sides of $T_2$). Define $T$ to be the union of all such $Q$'s that intersect $T_1\cap T_2$, and thus $T$ is a rectangle and trivially satisfies $T_1\cap T_2\sub T$.
    
    Now given $x\in T$. Then there is some $Q\sub T$ such that $x\in Q$, and $Q\cap T_1\neq \varnothing$. Since $Q$ has side length $l_1\le 2$, we see that $x\in (1+2\sqrt 2)T_1$. Also, $Q\cap T_2\neq \varnothing$, and since $l_1\le l_2$ we see that $x\in 3T_2$. Thus we have $T\sub (1+2\sqrt 2)T_1\cap (1+2\sqrt 2)T_2$. 
\end{enumerate} 
Combining the reduction in the first part we see that we lose a factor at most $9\times (1+2\sqrt 2)<100$. In particular, the proposition holds.
\end{proof}

\subsubsection{Admissible partitions and bounded overlap}\label{sec_Yang_cover}
We prove the following proposition, which completes the proof of the bounded overlap in Proposition \ref{prop_P_y_large}. 

Recall the setup: $P:\R^2\to \R$ is a polynomial of degree at most $d$ and with bounded coefficients. Let $\Gd>0$ and $g$ be the smooth function with bounded derivatives of all orders as in Proposition \ref{prop_P_y_large} where $P(x,g(x))=0$ over $[-1,1]$. 

\begin{prop}\label{prop_bounded_overlap}
Suppose $\mathcal I_\Gd$ is an admissible partition of $[-1,1]$ for $g$ at the scale $\Gd$ such that for each $I\in \mathcal I_\Gd$ we also have \eqref{eqn_maximally_flat}. Then for each absolute constant $C\ge 1$, the intervals $\{CI:I\in \mathcal I_\Gd\}$ have bounded overlap.
\end{prop}

To this end, we need the following lemma.
\begin{lem}\label{lem_appendix_flat}
For any interval $I\sub [-1,1]$, the following are equivalent.
\begin{enumerate}
    \item \label{item_linear} $g$ is $O(\Gd)$-flat over $I$, that is,
    $$
    \sup_{x,x_0\in I}|g(x)-g(x_0)-g'(x_0)(x-x_0)|\lesssim\Gd.
    $$
    \item \label{item_quadratic} We have
    \begin{equation*}
     \sup_{x\in I} |g''(x)||I|^2\lesssim \Gd. 
    \end{equation*}
    \item We have \label{item_linear2}
    \begin{equation*}
        \sup_{x,x_0\in I}|g'(x)-g'(x_0)||I|\lesssim \Gd.
    \end{equation*}
    \item \label{item_CI} For any absolute constant $C\geq 1$, $g$ is $O(\Gd)$-flat over $CI\cap [-1,1]$.
\end{enumerate}
\end{lem}

\begin{proof}
\begin{itemize}
    
    \item ``$\eqref{item_linear}\implies \eqref{item_quadratic}, \eqref{item_CI}$" Taking $x_0$ to be the centre of $I$, we have
$$
\sup_{x\in I}|g(x)-g(x_0)-g'(x_0)(x-x_0)|\lesssim \Gd.
$$
Since $g'$ is bounded, by a translation and a shear transform we may assume $g(x_0)=g'(x_0)=0$. Thus $|g(x)|\lesssim \Gd$ over $I$. Write
$$
P(x,y)=A(x)+yB(x,y),
$$
for polynomials $A,B$ with bounded coefficients. Then we have
$$
0=P(x,g(x))=A(x)+g(x)B(x,g(x)).
$$
Since $B$ has bounded coefficients and $g=O(\Gd)$, we have $A=O(\Gd)$ over $I$. Since $A$ is a polynomial, by Corollary \ref{cor_polycoeff} and a simple rescaling we have $|A'(x)||I|\lesssim \Gd$ over $I$. The same reasoning gives $|A''(x)||I|^2\lesssim \Gd$ over $I$.

By direct computation, we have
$$
g''(x)=\frac{P_{xx}P_y^2+P_{x}^2P_{yy}-2P_xP_yP_{xy}}{P_y^3}(x,g(x)).
$$
Thus over $I$ we have
$$
P_x(x,g(x))|I|= A'(x)|I|+g(x)B_x(x,g(x))|I|=O(\Gd).
$$
Similarly, we can show $P_{xy}(x,g(x))|I|=O(\Gd)$ and $P_{xx}(x,g(x))|I|^2=O(\Gd)$ over $I$. Using $P_y\gtrsim 1$, we thus have $g''|I|^2=O(\Gd)$ over $I$, and thus \eqref{item_quadratic} is true.

To prove \eqref{item_CI}, note that by Corollary \ref{cor_polycoeff}, $A=O(\Gd)$ over $I$ implies $A=O(\Gd)$ over $CI$. Following the same argument as above then shows $g''|I|^2=O(\Gd)$ over $CI\cap [-1,1]$. Then \eqref{item_CI} follows from Taylor's theorem.

\item ``$\eqref{item_quadratic}\implies \eqref{item_linear2}\implies \eqref{item_linear}$" This part is an immediate consequence of Taylor's theorem.

\item ``$\eqref{item_CI}\implies \eqref{item_linear}$" This part is trivial.
\end{itemize}
\end{proof}

\begin{proof}[Proof of Proposition \ref{prop_bounded_overlap}]

By our definition of bounded overlap, we will show that every $x\in \R$ lies in at most $O(1)$ many intervals $CI$. Fix $x\in \R$ and assume there is a collection $\mathcal B$ of intervals $I\in \mathcal I_\Gd$ such that $x\in CI$. For simplicity of notation, we write $x=0$. Also, for an interval $I$ we use $c_I$ to denote its centre.

Suppose there are $N$ intervals to the right of $0$ that lie in $\mathcal B$. List them as $I_n$, $1\le n\leq N$ from left to right. Then 
$$
c(I_n) \leq C|I_n|/2.
$$
On the other hand, by disjointness of the intervals $I_n$, we have
$$
c(I_n) \geq  \sum_{k=1}^{n-1} |I_k|.
$$
Thus we have for every $1\leq n \leq N$,
$$
|I_n| \geq 2C^{-1}\sum_{k=1}^{n-1} |I_k| 
$$
Therefore
$$
|I_N| \geq 2C^{-1}\sum_{k=1}^{N-1} |I_k| \geq 2C^{-1} (1+2C^{-1})\sum_{k=1}^{N-2}|I_k|\geq ... \geq 2C^{-1}(1+2C^{-1})^{N-2}|I_1|.
$$

Since $g$ is $\Gd$-flat over $I_N$, using Part \eqref{item_quadratic} of Lemma \ref{lem_appendix_flat} then shows that
\begin{equation*}
   \sup_{x\in I_N}|g''(x)||I_1|^2\lesssim \Gd (1+2C^{-1})^{-N}. 
\end{equation*}
By Part \eqref{item_CI} of Lemma \ref{lem_appendix_flat}, we also have
\begin{equation*}
   \sup_{x\in CI_N\cap [-1,1]}|g''(x)||I_1|^2\lesssim \Gd (1+2C^{-1})^{-N}. 
\end{equation*}
But since $CI_N\cap [-1,1]$ contains $I_1$, we have
\begin{equation*}
   \sup_{x\in I_1}|g''(x)||I_1|^2\lesssim \Gd (1+2C^{-1})^{-N}.
\end{equation*}
Using \eqref{eqn_maximally_flat} for $I_1$ we thus have $N=O_C(1)$. The argument on the left hand side of $0$ is the same, and hence the result follows.

\end{proof}

\section{Proof of uniform decoupling for bivariate polynomials}\label{sec_proof_main}
Assuming Theorem \ref{thm_small_hessian} for the moment, we are ready to prove Theorem \ref{thm_main_uniform_decoupling_eps}. The main idea is a Pramanik-Seeger type iteration. We first construct a suitable cover of $[-1,1]^2$ adapted to $\phi$, and then prove a corresponding decoupling inequality at the end.

Again, all implicit constants in this section are assumed to depend on the degree $d$.

\subsection{The base case}
Let $M=M(d,\Ge)$ be a large constant to be determined. 

Given $0<\Gd<1$. If $\Gd\ge M^{-3}$ then we simply take $\mathcal P_\Gd$ to be the partition of $[-1,1]^2$ into squares of side length $c\Gd^{1/2}$, where $c=c(d)$ is chosen such that each square is $(\phi,\Gd)$-flat. The decoupling inequality is then trivial, since the constant is allowed to depend on both $d,\Ge$. 

\subsection{Finding the cover}\label{sub_cover}
Now let $\Gd<M^{-3}$. Our first step will be finding a cover of $[-1,1]^2$ by $(\phi,\Gd)$-flat rectangles. The overlap and the decoupling inequality will be discussed later.

Recall the exponent $\Ga=\Ga(d)>0$ in Theorem \ref{thm_small_hessian}. Our induction hypothesis is that the conclusion holds at all coarser scales $\Gd'\gtrsim M^\Ga\Gd$.

%\underline{Step 1: Curved part and flat part.}
\subsubsection{Curved part and flat part}\label{subsub_curved_flat}
We first study the Hessian determinant $P:=\det D^2\phi$. Partition $[-1,1]^2$ into two parts $S_{\text{curved}}$ and $S_{\text{flat}}$ as follows:
\begin{align*}
    &S_{\text{curved}}:=\{(x,y)\in [-1,1]^2: |P(x,y)|>M^{-1}\}\\
    &S_{\text{flat}}:=\{(x,y)\in [-1,1]^2: |P(x,y)|\leq M^{-1}\}.
\end{align*}
Consider the partition of $[-1,1]^2$ by squares of side length $\sim \Gd^{1/2}$ so that each of them is $(\phi,\Gd)$-flat, and denote by $\mathcal C$ the collection of those squares $Q$ that intersect $S_{\text{curved}}$. Note that for each $Q\in \mathcal C$, we have $|P(x,y)|>M^{-1}/2$ on $Q$ since we assume $\Gd^{1/2}\ll M^{-1}$ and $\nabla P$ is bounded. If we can find a cover $\mathcal F$ of $S_{\text{flat}}$ by $(\phi,\Gd)$-flat rectangles, then $\mathcal C\cup \mathcal F$ will be the required cover $\mathcal P_\Gd$.

%\underline{Step 2: Covering the flat part by intermediate rectangles.}
\subsubsection{Covering the flat part by intermediate rectangles}\label{subsub_intermediate}
To cover the flat part, we will first apply Theorem \ref{thm_2D_general_uniform} above to $P$ at $M^{-1}$ to cover $S_{\text{flat}}$ by rectangles $T$ on each of which $|P|\lesssim M^{-1}$. Also, the rectangles $T$ have bounded overlap.

%\underline{Step 3: Rescaling.}
\subsubsection{Rescaling}\label{subsub_rescaling_1}
It remains to further decompose each $T$ into smaller rectangles on each of which $\phi$ is $\Gd$-flat. To do this, we first apply a rescaling as follows.

Denote by $L_2$ the composition of a rotation and a rescaling such that $L_2$ is a bijection from $[-1,1]^2$ to $T$. Equivalently, we may cover $[-1,1]^2$ by boundedly overlapping rectangles on each of which $\phi_1:=\phi\circ L_2$ is $\Gd$-flat, and then apply the inverse of $L_2$. Note that $\phi_1$ has the same degree as $\phi$, and $\sup_{[-1,1]^2}|\phi_1|= \sup_{T}|\phi|$. Also, recalling that $|\det D^2 \phi|\lesssim M^{-1}$ on $T$, we have $|\det D^2 \phi_1|\lesssim M^{-1}$ over $[-1,1]^2$ since $|\det L_2|\leq 1$. By Proposition \ref{prop_polycoeff} we know that all coefficients of $\det D^2\phi_1$ are essentially bounded by $M^{-1}$.

\subsubsection{Applying small Hessian theorem}\label{subsub_small_hessian}
We now apply Theorem \ref{thm_small_hessian} to $\phi_1$ to find a suitable rotation $\rho$ such that
$$
\phi_2(x,y):=\phi_1\circ \rho(x,y)=A(x)+M^{-\Ga} B(x,y),
$$
where $A,B$ have bounded coefficients. It then suffices to cover $[-1,1]^2$ by $(\phi_2,\Gd)$-flat rectangles.

\subsubsection{Applying 2D uniform decoupling}\label{subsub_2D_uniform}
%\underline{Step 1: Applying 2D uniform decoupling.}
Now we do a cylindrical decoupling at the scale $M^{-\Ga}$. More precisely, we first apply Theorem \ref{thm_Yang2_rational} to the univariate polynomial $A$ to partition $[-1,1]$ into intervals $I$ on each of which $A$ is $\sim M^{-\Ga}$-flat. If $l$ denotes the linear approximation of $A$ at the centre of $I$, then we have $\sup_I |A(x)-l(x)|\lesssim M^{-\Ga}$.

\subsubsection{Rescaling}\label{subsub_rescaling_2}
%\underline{Step 2: Rescaling.}
It remains to cover each such strip $I\times [-1,1]$ by $(\phi_2,\Gd)$-flat rectangles. To achieve this, we first denote by $L_1$ a linear bijection from $[-1,1]$ to $I$, and consider $\phi_2(L_1 x,y)$. Then the polynomial $\phi_2(L_1x,y)-l(L_1x)$ is essentially bounded by $M^{-\Ga}$ in $[-1,1]^2$. By Proposition \ref{prop_polycoeff}, all of its coefficients are essentially bounded by $M^{-\Ga}$. We then define
$$
\phi_3(x,y)=M^{\Ga}(\phi_2(L_1 x,y)-l(L_1 x)),
$$
which has bounded coefficients. Thus it suffices to cover $[-1,1]^2$ by $(\phi_3,\sim M^\Ga\Gd)$-flat rectangles.

\subsubsection{Applying the induction hypothesis}\label{subsub_induction_hypothesis}
%\underline{Step 3: Applying induction hypothesis.}
We can now use the induction hypothesis to $\phi_3$ at scale $\sim M^\Ga\Gd$ to cover $[-1,1]^2$ by rectangles prescribed in this theorem. Reversing all linear transformations in the previous steps, we have obtained a cover $\mathcal F$ of $[-1,1]^2$ by $(\phi,\Gd)$-flat rectangles, as required.

\subsection{Decoupling inequality}
Having obtained the cover $\mathcal P_\Gd$ at each scale $\Gd$, we now proceed to proving the corresponding decoupling inequality.

Denote by $D(\Gd)$ the smallest constant such that the following inequality holds for every $f$ Fourier supported in $\mathcal N^\phi_{\Gd}([-1,1]^2)$:
\begin{equation}
    \norm {f}_{L^p(\R^3)}\leq D(\Gd)\# \mathcal P_\Gd^{\frac 1 2 - \frac 1 p}\left(\sum_{P\in \mathcal P_\Gd}\norm{f_P}_{L^p(\R^3)}^p\right)^{\frac 1 p}.
\end{equation}
Our goal is to show that $D(\Gd)\leq C_{d,\Ge}\Gd^{-\Ge}$ independent of the choice of $\phi$. To do this, we carefully follow the steps in Section \ref{sub_cover}.

\subsubsection{Curved part: Bourgain-Demeter decoupling}\label{subsub_curved_decoupling}
The decoupling for the curved part in \ref{subsub_curved_flat} is an easy consequence of the standard Bourgain-Demeter decoupling, as in Theorem 1.1 of \cite{BD2017}. That is, for every function Fourier supported in $\mathcal N^\phi_\Gd(\cup \mathcal C)$ we have
\begin{equation}\label{eqn_curved}
    \norm {f}_{L^p(\R^3)}\lesssim_{M,\Ge} \Gd^{-\Ge}\# \mathcal C^{\frac 1 2 - \frac 1 p}\left(\sum_{Q\in \mathcal C}\norm{f_Q}_{L^p(\R^3)}^p\right)^{\frac 1 p}\leq \Gd^{-\Ge}\# \mathcal P_\Gd^{\frac 1 2 - \frac 1 p}\left(\sum_{P\in \mathcal P_\Gd}\norm{f_P}_{L^p(\R^3)}^p\right)^{\frac 1 p}.
\end{equation}
Notice that the decoupling constant $C_{d,M,\Ge}$ can be chosen to be independent of the choice of $\phi$; the proof of this assertion is elementary, and the reader may refer to Section 5 of \cite{Yang2} for a rigorous proof in the 2D case.

In view of the bounded overlap to be established in Section \ref{sec_overlap_proof_main} below, by the triangle and H\"older's inequalities it remains to prove the decoupling inequality for the flat part.

\subsubsection{Intermediate decoupling}\label{subsub_intermediate_decoupling}
Following Theorem \ref{thm_2D_general_uniform} as in Section \ref{subsub_intermediate} and applying Lemma \ref{lem_projection} with $\mathcal A=\mathcal T_{M^{-1}}$ and $\mathcal B=\{\mathcal N^{\phi}_{M^{-1}}(T):T\in \mathcal T_{M^{-1}}\}$, for every function $f$ Fourier supported in $\mathcal N^\phi_\Gd(S_{\text{flat}})$ we have
\begin{equation}\label{eqn_intermediate}
    \norm {f}_{L^p(\R^3)}\lesssim_{\Ge} M^{\Ge}\left(\sum_{T\in \mathcal T_{M^{-1}}}\norm{f_T}_{L^p(\R^3)}^2\right)^{\frac 1 2}.
\end{equation}
Note that the constant is also independent of the choice of $\phi$.

\subsubsection{2D uniform cylindrical decoupling}\label{subsub_2D_uniform cylindrical}
The rescaling and rotation in \ref{subsub_rescaling_1} and \ref{subsub_small_hessian} do not affect the decoupling inequality. The 2D uniform decoupling inequality in \ref{subsub_2D_uniform} implies the following: for every function $f$ Fourier supported in $\mathcal N^{\phi_2}_{O(M^{-\Ga})}([-1,1]^2)$, we have
\begin{equation}\label{eqn_2D_uniform}
    \norm {f}_{L^p(\R^3)}\lesssim_{\Ge} M^{\Ge\Ga}\left(\sum_{I}\norm{f_{I\times \R}}_{L^p(\R^3)}^2\right)^{\frac 1 2}.
\end{equation}
The constant is again independent of the choice of $\phi$.

\subsubsection{Applying the induction hypothesis}
The rescaling in \ref{subsub_rescaling_2} does not affect the decoupling inequality. From the induction hypothesis in \ref{subsub_induction_hypothesis} we have the following: for every function $f$ Fourier supported in $\mathcal N^{\phi_{3,I}}_{\Gd'}([-1,1]^2)$ (note that $\phi_3$ depends on individual choices of $I$) where $\Gd'\sim M^\Ga \Gd$, we have
\begin{equation}\label{eqn_induction_hypothesis}
    \norm {f}_{L^p(\R^3)}\leq D(\Gd')\#\mathcal P^{\frac 1 2 - \frac 1 p}_{\Gd'}\left(\sum_{P'\in \mathcal P_{\Gd',I}}\norm{f_{P'}}_{L^p(\R^3)}^p\right)^{\frac 1 p}.
\end{equation}
Here, $\mathcal P_{\Gd',I}$ is the cover of $[-1,1]^2$ by $(\phi_{3,I},\Gd')$-flat rectangles as constructed in Section \ref{subsub_induction_hypothesis}. Equivalently, by a rescaling we have
\begin{equation}
    \norm{f_{I\times \R}}_{L^p(\R^3)}\leq D(\Gd')\#\tilde{\mathcal P}^{\frac 1 2 - \frac 1 p}_{\Gd',I}\left(\sum_{\tilde P'\in \tilde {\mathcal P}_{\Gd',I}}\norm{f_{\tilde P'}}_{L^p(\R^3)}^p\right)^{\frac 1 p},
\end{equation}
where $\tilde{\mathcal P}_{\Gd,I}$ is a cover of $I\times [-1,1]$ by $(\phi_2,\Gd)$-flat rectangles and $\#\mathcal P_{\Gd'}=\#\tilde{\mathcal P}_{\Gd,I}$. Note that $\cup_I \tilde{\mathcal P}_{\Gd,I}$ is a cover of $[-1,1]^2$ by $(\phi_2,\Gd)$-flat rectangles.

\subsubsection{Combining decoupling estimates}\label{sub_combining_decoupling}
Now we combine the inequalities \eqref{eqn_2D_uniform} and \eqref{eqn_induction_hypothesis} using H\"older's inequality to the sum over $I$. More precisely, let $f$ be Fourier supported in $\mathcal N^{\phi_2}_{O(M^{-\Ga})}([-1,1]^2)$. Then
\begin{align*}
    \norm {f}_{L^p(\R^3)}
    &\lesssim_{\Ge} M^{\Ge\Ga}\left(\sum_{I}\norm{f_{I\times \R}}_{L^p(\R^3)}^2\right)^{\frac 1 2}\\
    &\leq M^{\Ge\Ga}D(\Gd')\left(\sum_{I}\#\tilde{\mathcal P}^{2\left(\frac 1 2 - \frac 1 p\right)}_{\Gd',I}\left(\sum_{\tilde P'\in \tilde {\mathcal P}_{\Gd',I}}\norm{f_{\tilde P'}}_{L^p(\R^3)}^p\right)^{\frac 2 p}\right)^{\frac 1 2}\\
    &\leq M^{\Ge\Ga}D(\Gd')\left[\left(\sum_{I}\left(\#\tilde{\mathcal P}^{2\left(\frac 1 2 - \frac 1 p\right)}_{\Gd',I}\right)^{\frac p {p-2}}\right)^{1-\frac 2 p}\left( \sum_{I}\left(\sum_{\tilde P'\in \tilde {\mathcal P}_{\Gd',I}}\norm{f_{\tilde P'}}_{L^p(\R^3)}^p\right)^{\frac 2 p\cdot \frac p 2}\right)^{\frac 2 p}\right]^{\frac 1 2}\\
    &=M^{\Ge\Ga}D(\Gd')\left[\left(\sum_I \#\tilde {\mathcal P}_{\Gd',I}\right)^{1-\frac 2 p}\left(\sum_I \sum_{\tilde P'\in \tilde {\mathcal P}_{\Gd',I}}\norm{f_{\tilde P'}}_{L^p(\R^3)}^p\right)^{\frac 2 p} \right]^{\frac 1 2}\\
    &=M^{\Ge\Ga}D(\Gd')[\# (\cup_I \tilde{\mathcal P}_{\Gd,I})]^{\frac 1 2-\frac 1 p}\left(\sum_{\tilde P'\in \cup_I \tilde{\mathcal P}_{\Gd,I}}\norm{f_{\tilde P'}}_{L^p(\R^3)}^p\right)^{\frac p 2}.
\end{align*}
Now we can reverse the rescaling and rotation from $\phi_2$ to the original $\phi$. The above decoupling inequality translates to the following: for each $T$ in the flat part $\mathcal T_{M^{-1}}$ we have
\begin{equation}
    \norm{f_T}_{L^p(\R^3)}\lesssim_\Ge M^{\Ge\Ga}D(\Gd') (\# \mathcal P_{\Gd,T})^{\frac 1 2-\frac 1 p}\left(\sum_{P\in \mathcal P_{\Gd,T}}\norm{f_{P}}_{L^p(\R^3)}^p\right)^{\frac p 2},
\end{equation}
where $\mathcal P_{\Gd,T}$ is the cover of $T$ by $(\phi,\Gd)$-flat rectangles $P$.

After combining with \eqref{eqn_intermediate} and using another H\"older's inequality the same way as right above, we have for every $f$ Fourier supported in $\mathcal N^\phi_{\Gd}(S_{\text{flat}})$
\begin{equation}
    \norm f_{L^p(\R^3)}\lesssim_\Ge M^{2\Ge}D(\Gd')\# \mathcal F^{\frac 1 2 - \frac 1 p}\left(\sum_{P\in \mathcal F}\norm{f_{P}}_{L^p(\R^3)}^p\right)^{\frac 1 p},
\end{equation}
where $\mathcal F=\cup_T \mathcal P_{\Gd,T}$.

Combining with the estimate \eqref{eqn_curved} of the curved part in \ref{subsub_curved_flat}, we thus obtain the following bootstrap inequality:
\begin{equation}\label{eqn_bootstrap}
    D(\Gd)\leq C_{M,\Ge}\Gd^{-\Ge}+C_\Ge D(\Gd')M^{2\Ge},
\end{equation}
where $\Gd'\sim M^\Ga \Gd$.

\subsubsection{Induction on scales}
Now it's time to iterate \eqref{eqn_bootstrap}. For simplicity of argument we take $\Gd'=M^\Ga \Gd$.

Iterating inequality \eqref{eqn_bootstrap} for $N$ times, we obtain
\begin{align*}
     D(\Gd)&\leq  C_{M,\Ge}\Gd^{-\Ge}[1+C_\Ge M^{(2-\Ga)\Ge}+\cdots+(C_\Ge M^{(2-\Ga)\Ge})^{N-1}]+(C_\Ge M^{2\Ge})^N D(M^{\Ga N} \Gd)\\
     &\leq \left (C_{M,\Ge}\Gd^{-\Ge}+D(M^{\Ga N} \Gd)\right )(C_\Ge M^{2\Ge})^N.
\end{align*}
We will stop this iteration once we reach $M^{\Ga N}\Gd\sim M^{-3}$, that is, when 
$$
N\sim \frac {\log \Gd^{-1}}{\Ga\log M}.
$$
In this case we have $D(M^{\Ga N} \Gd)\sim_{M,\Ge} 1$, and 
$$
(C_\Ge M^{2\Ge})^N\sim (\Gd^{-1})^{2\Ge+\frac {\log C_\Ge}{\Ga \log M}}.
$$
Finally, we are ready to specify the choice of $M=M(d,\Ge)$ mentioned at the beginning of the base case: we let $\frac{\log C_\Ge}{\Ga \log M}=\Ge$, that is,
$$
M=C_\Ge^{\frac 1 {\Ga \Ge}}.
$$
In this way, we have $D(\Gd)\lesssim_{\Ge}\Gd^{-4\Ge}$, independent of the choice of $\phi$. Since $\varepsilon$ is arbitrary, we have finished the proof of the decoupling inequality.

\subsection{Analysis of overlap}\label{sec_overlap_proof_main}
To show $\mathcal P_\Gd$ has $\Ge$-bounded overlap, we may actually prove a slightly stronger statement that $\sum_{T\in \mathcal P_\Gd}1_{100T}\leq C_{d,\Ge}\Gd^{-\Ge}$.

There are two main sources of overlaps between rectangles in $\mathcal P_\Gd$, namely, overlaps between rectangles created within one iteration, and overlaps between different iterations.

The former is easy to deal with. First, Section \ref{subsub_intermediate} gives rise to intermediate rectangles with their 100-dilations boundedly overlapping when we apply Theorem \ref{thm_2D_general_uniform}. Second, although Section \ref{subsub_2D_uniform} gives rise to a partition of $[-1,1]^2$ into strips of the form $I\times [-1,1]$, where the covering of each strip may slightly exceed $I\times [-1,1]$, the overlaps of the 100-dilations of the covering rectangles are also bounded, in view of Theorem \ref{thm_Yang2_rational}. Hence, the 100-dilations of the rectangles created within one iteration have bounded overlap.

For the latter kind of overlap, we use a very rough bound. Namely, since in each iteration the number of overlap is at most $C_d$, the final number of overlap between is at most 
$$
C_d^{N}\lesssim C_d^{\frac {\log \Gd^{-1}}{\Ga\log M}}\sim_\varepsilon \Gd^{-\Ge},
$$
as required.

We remark that a more refined analysis of the geometry of the rectangles will probably give a better overlap bound, which we do not pursue in this article.

\subsection{The convex case}\label{sub_convex}
Lastly, we prove a $\ell^2(L^p)$ decoupling inequality \eqref{eqn_decoupling_l2_eps} when $\det D^2 \phi\ge -\Gd^2$ over $[-1,1]^2$.

We claim that the set $Z_\Gd:=\{(x,y)\in [-1,1]^2:\det D^2\phi(x,y)\in [-\Gd^2,0]\}$ lies in the flat part of each iteration. First, it is easy to see that $Z_\Gd$ is a subset of $S_{\text{flat}}$ in Section \ref{subsub_curved_flat} since $\Gd^2<M^{-1}$. Next, note that after one iteration, the rescaled function $\phi_3$ in Section \ref{subsub_rescaling_2} satisfies $|\det D^2\phi_3|\leq M^{2\Ga}\Gd^2$, which is also much smaller than $M^{-1}$, over the corresponding rescaling of $Z_\delta$. Since the iteration stops after $N\sim \frac{\log\Gd^{-1}}{\Ga \log M}$ steps, we have for all $j$ (where $\phi_3$ gives rise to $\phi_4,\phi_5,\phi_6$ in the second iteration, and so on)
$$
|\det D^2\phi_{3j}|\leq M^{2N\Ga}\Gd^2\sim M^{-6}<M^{-1}.
$$
Thus the set $Z_\Gd$ is a subset of the flat part in every iteration.

We recall the simple fact that the composition of a convex function with any orientation-preserving affine map is also convex. As a result, by the argument above, in the $j$-th iteration, the function $\phi_{3j}$ is convex over $[-1,1]^2\setminus Z_\delta$, which is the curve part that we need to decouple further. Thus, in Section \ref{subsub_curved_decoupling} when we invoke Theorem \ref{thm_Bourgain_Demeter} we can obtain $\ell^2(L^p)$ estimates. Since all other decoupling inequalities, namely, the intermediate decoupling in Section \ref{subsub_intermediate_decoupling} and 2D cylindrical decoupling in Section \ref{subsub_2D_uniform cylindrical}, are $\ell^2(L^p)$ inequalities, in the end we are able to obtain a $\ell^2(L^p)$ decoupling inequality for the polynomial $\phi$.

\section{Polynomials with small Hessian determinant}\label{sec_small_hessian}
In this section, we prove Theorem \ref{thm_small_hessian}.  All implicit constants in this section are assumed to depend on the degree $d$.

\subsection{A reduction}In this subsection, we reduce Theorem \ref{thm_small_hessian} to studying a special type of polynomials. We need the following lemma about eigenvalues of Hessian matrix of polynomials.

\begin{lem}\label{lem_large_eigenvalue}
Let $\phi:\R^2\to \R$ be a polynomial of degree at most $d$, without linear terms. Suppose that $\phi$ has $O(1)$ coefficients and at least one of them is $\sim 1$. Then there exists $(x_0,y_0)\in B(0,1)$ such that one of the eigenvalues of $D^2 \phi(x_0,y_0)$ is of magnitude $\sim 1$.

\end{lem}

\begin{proof}
By a rotation and replacing $\phi$ by $-\phi$ if necessary, we may assume that $P_{xx}$ has at least one coefficient $ \sim 1$. By Proposition $\ref{prop_polycoeff}$ again, there exists $(x_0,y_0) \in B(0,1)$ such that $P_{xx}(x_0,y_0) \sim 1$.

Note that $\phi_{xx}+\phi_{yy}$ is the trace, and hence the sum of eigenvalues, of $D^2\phi$. Thus, if $|\phi_{xx}(x_0,y_0)+\phi_{yy}(x_0,y_0)| > \phi_{xx}(x_0,y_0) /2 \sim 1$, we are done. Otherwise, we have $\phi_{yy}(x_0,y_0) < -\phi_{xx}(x_0,y_0)/2<0$. However, $$\det D^2 \phi(x_0,y_0) < -\phi_{xx}^2(x_0,y_0)/2 - \phi_{xy}^2(x_0,y_0).$$ Therefore $$|\det D^2 \phi(x_0,y_0)| \geq \phi_{xx}^2(x_0,y_0)/2 \sim 1.$$ 
This implies that the product of the eigenvalues of $D^2\phi(x_0,y_0)$ is bounded below. Since both eigenvalues are $O(1)$, they have magnitude $\sim 1$, as desired.
\end{proof}

Lemma \ref{lem_large_eigenvalue} helps us to reduce Theorem \ref{thm_small_hessian} to the following case.

\begin{prop}\label{prop_small_hessian_assume_A}
There is a constant $\alpha = \alpha(d) \in (0,1]$ such that any polynomial $Q(x,y)$ satisfying 
\begin{enumerate}
    \item $\deg Q \leq d$;
    \item $Q$ has no linear term;
    \item the only second order term of $Q$ is $x^2$;
    \item all coefficients of $Q$ are $O(1)$;
    \item all coefficients of $\det D^2 Q $ are bounded by some $\nu \in (0,1)$;
\end{enumerate}
is of the form
\begin{equation}\label{eq:prop_small_hessian_assume_A}
    Q(x,y) = \Tilde{A}(x) + \nu^{\alpha}\Tilde{B}(x,y),
\end{equation}
where $\Tilde{A},\Tilde{B}$ are polynomials with $O(1)$ coefficients.
\end{prop}

We will prove Proposition \ref{prop_small_hessian_assume_A} in the next subsection.

\begin{proof}[Proof of Theorem \ref{thm_small_hessian} assuming Proposition \ref{prop_small_hessian_assume_A}] Let $\phi$ be a polynomial satisfying the assumption of Theorem \ref{thm_small_hessian}. By dividing $\phi$ by its largest coefficient, we may assume without loss of generality that $\phi$ has some coefficient $1$. 

Now, we apply Lemma \ref{lem_large_eigenvalue} and get $(x_0,y_0) \in B(0,1)$ such that one of the eigenvalues $\lambda_1$ of $D^2\phi(x_0,y_0)$ is of magnitude $\sim 1$. By dividing $\phi$ by $\lambda_1$ and replacing $\phi$ by $-\phi$ if necessary, we assume without loss of generality that $\lambda_1=1$.

Since $\sup_{B(0,1)}|\det D^2 \phi | \lesssim \nu$ by Proposition \ref{prop_polycoeff}, the other eigenvalue $\lambda_2$ of $D^2\phi(x_0,y_0)$ is of magnitude $O(\nu)$. Let $\rho$ be a rotation that sends $ \{e_1,e_2\}$ to the unit eigenvectors of $D^2\phi(x_0,y_0)$ and  $\tau$ be a translation that sends  the origin to $\rho^{-1}(x_0,y_0)$. By the eigenvalue analysis, the second order terms of $\Tilde{\phi} = \phi \circ \rho \circ \tau$  are given by $x^2 + O(\nu)y^2$. Define $Q$ by removing all the terms of degree at most two except $x^2$ in $\Tilde{\phi}$. Then we see that all assumptions in Proposition \ref{prop_small_hessian_assume_A} are satisfied. Hence, $Q$ is of the form \eqref{eq:prop_small_hessian_assume_A}.

Now, we analysis the coefficients of $\phi\circ \rho = \Tilde{\phi} \circ \tau^{-1}$. First, $Q \circ \tau^{-1}$ is also of the form \eqref{eq:prop_small_hessian_assume_A}. Moreover, the linear terms in $\Tilde{\phi}$ have no impact on the higher order terms in $\Tilde{\phi}\circ \tau^{-1}$ and $O(\nu)y^2$ can be absorbed to $\Tilde{B}$. On the other hand, $\phi$ has no linear term, and so does $\phi \circ \rho$. In conclusion, we see that 
$$
\phi \circ \rho (x,y) = A(x) + \nu^{\alpha}B(x,y)
$$
for some polynomials $ A, B$ with $O(1)$ coefficients, as desired.
\end{proof}

\subsection{The representing line} In this subsection, we prove Proposition \ref{prop_small_hessian_assume_A}. We first introduce the following terminology.
\begin{defn}
Let $\phi:\R^2\to \R$ be a polynomial. We consider the set of nonzero indices of $\phi$, namely,
$$
N(\phi):=\{(m,n)\in \Z_{\geq 0}^2:\text{the coefficient of $x^m y^n$ in $\phi$ is nonzero}\}.
$$
Given a straight line $\ell\sub \R^2$ we denote by $\phi|_\ell$ the sum of terms of $\phi$ with indices lying on $\ell$.
\end{defn}

We have the following simple lemma.
\begin{lem}\label{lem_rep} Let $\ell$ be such that all points of $N(\phi)$ lie on one side of $\ell$, inclusive. Then all points of $N(\det D^2 \phi)$ lie on the same side of $\ell':=\ell + \ell -\{(2,2)\}$. Moreover, 
$$
(\det D^2 \phi)|_{\ell'} = \det D^2 (\phi|_{\ell}).
$$
\end{lem}

\begin{proof}
Write $\phi = \phi|_{\ell} + R$. Then
$$
\det D^2 \phi = \det D^2 (\phi|_\ell)+(\phi|_{\ell})_{xx}R_{yy} + (\phi|_{\ell})_{yy}R_{xx}-(\phi|_{\ell})_{xy}R_{xy} + \det D^2 R. 
$$
All the terms except $\det D^2 (\phi|_\ell)$ are strictly on one side of $\ell+\ell -\{(2,2)\}$.
\end{proof}

In view of Lemma \ref{lem_rep}, the boundaries of the convex hull of $N(\phi)$ are important in the analysis. Since we are interested in polynomials whose only second order term is $x^2$, the boundaries that contain the term $x^2$ are of particular interest. 

\begin{defn}
Let $Q$ be a polynomial that satisfies the five assumptions in Proposition \ref{prop_small_hessian_assume_A}. $\ell \subset \R^2$ is called a representing line of $Q$ if
\begin{enumerate}
    \item $\ell$ contains $(2,0)$ and at least one other point in $N(Q)$;
    \item $\ell$ is not horizontal;
    \item all points of $N(Q)$ lie on one side of $\ell$, inclusive. 
\end{enumerate}
\end{defn}

\begin{prop}\label{prop_line}
Let $Q$ be a polynomial that satisfies the five assumptions in Proposition \ref{prop_small_hessian_assume_A} and $\ell$ be a representing line of $Q$. Then 
$$
Q|_\ell = x^2 + \nu^{\beta} B(x,y),
$$
for some $\beta \in (0,1]$ depending only on $d$ and some polynomial $B$ with $O(1)$ coefficients.
\end{prop}

\begin{proof} Express $\ell$ in the $(m,n)$-plane by the equation $m=t n+2$ where $t\in \R$. By Lemma \ref{lem_rep}, we see that $\det D^2(Q|_\ell)$ has coefficients bounded by $\nu$.

\subsubsection{Case $t<0$}
Since we are in $\Z_{\ge 0}^2$ we see $Q|_\ell$ is either of the form $Q|_\ell(x,y)=x^2+a y^k$ where $k \geq 3$ or
$$
Q|_\ell(x,y)=x^2+a_1 xy^k+a_2 y^{2k}
$$
where $k \geq 2$. 

In the former case, $\det D^2(Q|_\ell)= 2 a k (k-1)y^{k-2}$. Then $a=O(\nu)$ and we get the form we want. 

In the latter case, a direct computation shows that
$$
\det D^2(Q|_\ell)=2a_1k(k-1) xy^{k-2}+[4a_2 k(2k-1)-a_1^2k^2]y^{2k-2}.
$$

Since $\det D^2 (Q|_\ell)$ has coefficients bounded by $\nu$, we have $a_1=O(\nu)$, $a_2=O(\nu)$ as desired.

\subsubsection{Case $t\geq 0$}
In this case, $\ell$ is either vertical or has a positive slope. Let $cx^{k_1}y^{k_2}$ be the highest order term of $Q|_\ell$.
If $k_2\geq 1$, then by direct computation,
$$
\det D^2 (Q|_\ell)=-c^2 k_1k_2 (k_1+k_2-1)x^{2k_1-2}y^{2k_2-2}+\text{lower order terms}.
$$
In particular, we see that $c=O(\nu^{1/2})$. Thus we can approximate $Q|_\ell$ by $\Tilde{Q}=Q|_\ell-cx^{k_1}y^{k_2}$ and see that all coefficients of $\det D^2 \Tilde{Q}$ are of the order $O(\nu^{1/2})$. Note that $\tilde{Q}$ satisfies the same five assumptions with $\nu$ replaced by $\nu^{1/2}$. Each such approximation reduces the degree by at least $1$. This process can be repeated at most $d$ times until we arrive at $x^2$. In conclusion, we see that $Q|_\ell$ is of the form $x^2+\nu^{\beta}B(x,y)$ where $B$ has bounded coefficients, and $\beta = 2^{2-d}$.

\end{proof}

We are now ready to prove Proposition \ref{prop_small_hessian_assume_A}, thus completing the proof of Theorem \ref{thm_small_hessian}.

\begin{proof}[Proof of Proposition \ref{prop_small_hessian_assume_A}]
Let $Q$ be a polynomial that satisfies the five assumptions. If there is no representing line of $Q$, we are done because in this case $Q$ is a function of $x$. Otherwise, by Proposition \ref{prop_line}, all coefficients on representing lines $\ell$ of $Q$ are $O(\nu^\beta)$. We can then approximate $Q$ by $Q_1 = Q- Q|_{\ell}+x^2$ and $\det D^2 Q_1$ has $O(\nu^\beta)$ coefficients. For $j\geq 1$, repeat the process with $\nu$ replaced $\nu^{\beta^{j}}$ and let $Q_{j+1} = Q_j - (Q_j)|_{\ell_j}$, for some representing lines $\ell_j$ of $Q_j$. The process will be terminated in less than $(d+1)^2$ steps when no representing lines are available. In summary, we see that all coefficients of $Q$, except those terms containing $x$ only, is $O(\nu^{\beta^{(d+1)^2}})$. Thus, we obtain \eqref{eq:prop_small_hessian_assume_A} with $\alpha = \beta^{(d+1)^2}$.
\end{proof}

\section{Appendix}\label{sec_appendix}

The following facts about polynomials will be used extensively, and are the key to many of our analysis of polynomials.
\begin{prop}\label{prop_polycoeff}
For any $n$-variate real polynomial $P=\sum_\Ga c_\Ga \xi^\Ga$ of degree at most $d$, we have the following relation:
$$
\sup_{ [-1,1]^n}|P|\sim_{n,d}\max_\Ga|c_\Ga|.
$$
As a result, we also have the following equivalences:
$$
\sup_{[-1,1]^n}|P|\sim_{n,d} \sum_\Ga|c_\Ga|\sim_{n,d,C}\sup_{ [-C,C]^n}|P|\sim_{n,d} \sup_{ B^n(0,C)}|P|,
$$
for any absolute constant $C$.
\end{prop}
For a proof of this proposition, the reader may consult \cite{Kellogg1928}.

\begin{cor}\label{cor_polycoeff}
For any rectangle $T \subset \R^n$ and any absolute constant $C$,
$$
\sup_{T} |P| \sim _{n,d,C} \sup_{CT} |P|.
$$
\end{cor}

\begin{proof}
Let $l$ be an affine transformation that maps $[-1,1]^n$ to $T$. Apply Theorem \ref{prop_polycoeff} to $P \circ l$ to get
$$
\sup_{[-1,1]^n} |P \circ l | \sim_{n,d,C} \sup_{[-C,C]^n} |P \circ l |, 
$$
which implies the desired relation.
\end{proof}

\bibliographystyle{empty}
\bibliography{sample}

\end{document}